\documentclass[11pt,a4paper]{amsart}

\usepackage{amssymb, amsmath, amsthm, amscd}

\usepackage{enumerate}

\usepackage[all]{xy}

\theoremstyle{plain} \newtheorem{theorem}{Theorem}[section]
\theoremstyle{plain} \newtheorem{corollary}[theorem]{Corollary}
\theoremstyle{plain} \newtheorem{proposition}[theorem]{Proposition}
\theoremstyle{plain}\newtheorem{lemma}[theorem]{Lemma}
\theoremstyle{definition} \newtheorem{definition}[theorem]{Definition}
\theoremstyle{definition}
\theoremstyle{remark}\newtheorem{remark}[theorem]{Remark}
\theoremstyle{definition}\newtheorem*{hypotheses}{Hypotheses}
\theoremstyle{definition}
\theoremstyle{remark}

\newcommand{\F}{{\mathcal{F}}}

\def\restriction#1#2{\mathchoice
              {\setbox1\hbox{${\displaystyle #1}_{\scriptstyle #2}$}
              \restrictionaux{#1}{#2}}
              {\setbox1\hbox{${\textstyle #1}_{\scriptstyle #2}$}
              \restrictionaux{#1}{#2}}
              {\setbox1\hbox{${\scriptstyle #1}_{\scriptscriptstyle #2}$}
              \restrictionaux{#1}{#2}}
              {\setbox1\hbox{${\scriptscriptstyle #1}_{\scriptscriptstyle #2}$}
              \restrictionaux{#1}{#2}}}
\def\restrictionaux#1#2{{#1\,\smash{\vrule height .8\ht1 depth .85\dp1}}_{\,#2}}
\numberwithin{equation}{section}

\title[Deformations and moduli of structures on manifolds]{Deformations and moduli of structures on manifolds: general existence theorem and application to the sasakian case}
\author{Laurent Meersseman and Marcel Nicolau}
\date{\today}
\thanks{This work was partially supported by the grants Marie Curie 271141 DEFFOL and MTM2011-26674-C02-01 from the Ministry of Economy and Competitiveness of Spain and by the grant UNAB10-4E-378 co-funded by ERDF  ``A way to build Europe". 
The first author enjoyed the warmful atmosphere of the CRM at Bellaterra during part of the preparation of this work. We also thank J.A. \'Alvarez-L\'opez for giving us a proof of Proposition \ref{density}.}


\address{Laurent Meersseman\\ LAREMA\\ Universit\'{e} d'Angers\\ F-49045 Angers Ce\-dex, France\\ laurent.meersseman@univ-angers.fr}

\address{Marcel Nicolau\\ Departament de Matem\`{a}tiques \\ Universitat Aut\`{o}noma de Barcelona \\ E-08193  Bellaterra (Barcelona), Spain \\
 nicolau@mat.uab.cat}

\begin{document}
\begin{abstract}
In this paper, we deal with the moduli problem for geometric structures on manifolds. We prove an existence theorem of a local moduli space in a very general setting. Then, to show the strength of this result, we
apply it to the case of sasakian and Sasaki-Einstein structures for which until now only partial results are known.
\end{abstract}

\maketitle

\tableofcontents

\section{Introduction}
\label{intro}
One of the central problems in the study of geometric structures on manifolds is the moduli problem, i.e. the classification of such structures up to isomorphism. One asks for the existence of a space with nice properties (e.g. manifold or analytic space, including infinite-dimensional ones) parametrizing the isomorphism classes. In many cases it is too much to ask for and such a global space does not have a natural structure of manifold, nor is the zero set of global equations in a manifold. However, the local situation is much better and easier to handle, at least when defining an appropriate notion of local moduli space.\smallskip

This local strategy was initiated by Kodaira-Spencer in describing the small deformations of complex structures on compact manifolds. Already in their seminal paper \cite{K-S}, by looking at deformations of tori, Hopf and Hirzebruch surfaces, they remarked that the global point of view is not tractable in general and that the local one needs a weaker notion of local parameter space.\smallskip

In this paper, we deal with the local point of view but consider more general structures than complex ones. Especially, we consider sasakian structures for which until now only partial results are known. \smallskip

To be more precise,
the aim of the article is twofold

\begin{itemize}
\item proving an existence theorem of a local moduli space for geometric structures in a very general setting.
\item applying this theorem to the case of sasakian and Sasaki-Einstein structures.
\end{itemize}

In the case of complex structures, the local moduli space provided by our general theorem coincides with the Kuranishi space. Apart from sasakian structures, it can also be applied to many other geometric structures that will be considered subsequently.\smallskip

Let us begin with making some general comments. Solving the local moduli problem for a certain class of structures supposes attacking and solving three different but related sides of the question.

\begin{itemize}
\item a conceptual side: one has to {\it define} precisely what is meant by a local moduli space. It must be general enough to allow an existence theorem under reasonable hypotheses, and natural enough so that such a result is meaningful.
\item a theoretical side: to prove and state an existence theorem, as well as associated results (as a rigidity result).
\item a practical side: to provide technical tools to compute the local moduli space on explicit examples.
\end{itemize}

In the case of complex structures, the conceptual side is the notion of versal space; the theoretical side is Kuranishi's Theorem that asserts that every compact complex manifold $X$ has a versal space (called here the Kuranishi space of $X$); and the practical side is given by the criteria for a deformation to be rigid, or complete, or versal
in terms of the Kodaira-Spencer map and of the cohomology with values in the sheaf of holomorphic tangent vectors.\smallskip

Let us emphasize that a versal space is not an ideal local moduli space encoding every class of structures close to the base one as a single point. Nor it is a local moduli space in the classical sense recalled in section \ref{localmoduli}. For example, if $X$ is the second Hirzebruch surface, then the Kuranishi space of $X$ is (the germ of) a disk $\mathbb D$ with $0\in\mathbb D$ encoding $X$ and all other points encoding $\mathbb P^1\times\mathbb P^1$. Hence this disk encodes only two different complex structures up to biholomorphism.\smallskip 

Indeed, the Kuranishi space encodes every class of structures close to the base one (it is {\it complete}), but the requirement of single coding is replaced by the fact that it is minimal amongst all complete analytic spaces (it is {\it versal} (also called semi-universal) deformation space, cf. \cite{K-S} or \cite{MK}). This minimality property means that this space is the best approximation in the category of analytic spaces of the local moduli stack (cf. section \ref{stacks}).\smallskip

 The notion of versal space is also adapted to the case of transversely holomorphic foliations, for which an existence theorem is proved in \cite{GHS} and in \cite{AzizMarcel} for the case with fixed differentiable foliated type.  Apart from these two classes and some variants, and although versality can be defined in a very general setting, in many other situations (CR structures, smooth foliations,  ... ), it is not true, or not known, whether a versal space exists.\smallskip

A strict application of Kodaira-Spencer-Kuranishi methods requires the existence of an elliptic resolution of the sheaf of infinitesimal automorphisms of the structure. This has strong consequences, in particular the finiteness of the dimension of the versal space. This prevents from using these techniques in many cases. Nevertheless, if we focus on the proof of Kuranishi's Theorem given in \cite{Kur1}, the arguments and strategy only involve very general tools that have nothing to do with the cohomological aspects of the classical deformation theory of Kodaira-Spencer. This explains that Kuranishi's approach is used by Ebin in \cite{Ebin} to prove the existence of a local moduli space for riemannian metrics, and by Donaldson for ASD hermitian connections in \cite{Donaldson}. Both are indeed global moduli spaces, benefiting from the fact that the isotropy groups are finite (at least for some generic choices). \smallskip

The strategy of this paper is to clean up Kuranishi's proof given in \cite{Kur1}, to keep from it only the scheme that was applied in Ebin's and Donaldson's cases, and to give general hypotheses under which it can be applied to. This supposes to be in a coherent general setting and to have a definition of a (generalized) local moduli space. Both comes from \cite{modules}, with slight modifications. Let us be more specific. Kuranishi's Theorem is a slice theorem and not a deformation theorem. At a fixed base point, one looks for a local slice to the action of the diffeomorphism group onto the space of complex operators. Versality can then be interpreted as a minimality condition of the slice: the orbit of the base point must intersect the slice only at a discrete set, see \cite{modules} and section \ref{minimal}. All this can easily be transposed to the abstract setting of a topological group acting on a Hilbert manifold.\smallskip

In this framework, one forgets about families and deformations, since one then tries to show the existence of a minimal slice using only the classical inverse function Theorem. We prove that there are basically two conditions to succeed. The first one is geometric and obvious: the local orbit of the base point must be closed. The second one is analytic and more surprising: the isotropy group must contain only more regular elements. To understand it, think of the case where the structures under consideration are sections of some smooth bundle. To run the program, one uses sections of Sobolev class, say $l$, (to be in a Banach space). Then, the condition is that the isotropy group contains only elements of class at least $l+1$. This requirement is a direct consequence of the fact that composition of diffeomorphisms of same Sobolev class is only a continuous and not a $C^1$ map. And it is obviously satisfied for complex structures (since automorphisms are holomorphic), riemannian metrics (since isometries are of class $C^\infty$), and generic ASD connections (since their isotropy group is plus or minus the identity). \smallskip

We consider two different settings. The most general one, developed in section \ref{existence}, in which the set of structures are modeled on arbitrary Banach manifolds and analytic spaces, and the second one, developed in section \ref{existencesmooth}, in which the set of structures is the set of sections of a bundle over the given manifold. The main results are gathered in Theorem \ref{main} for the first setting, and in Theorem \ref{mainsmooth} for the second one.\smallskip

To demonstrate the strength of our results, we show the existence of a Kuranishi type moduli space for sasakian and Sasaki-Einstein structures. A sasakian manifold is an odd-dimensional riemannian manifold endowed with a very special metric (cf. section \ref{sasakian}) and comes equipped with many induced additional structures: a strongly pseudo-convex CR structure, a contact form, a flow  by isometries, a transversely K\"ahler foliation. All these structures are closely linked one to each other. Hence there are several deformation problems one can state and study, depending on what is kept fixed or not. Some of them have been solved (cf. the reference book \cite{BoyerBook}), including the case of deformations of the transversely K\"ahler structure (the smooth foliation being fixed), which relies on the results of \cite{AzizMarcel}. But there is no comprehensive construction of a general local moduli space, general meaning that we allow deformations of all structures at the same time. \smallskip

This is what we do in Theorem \ref{mainsasakian}. In the same vein, we also construct in Theorem \ref{Kuretasasakien} a Kuranishi type moduli space for deformations of the contact form of a sasakian manifold. We also describe the relationship between these different Kuranishi type moduli spaces associated to a sasakian structure. Finally we deduce Kuranishi type spaces for deformations of Sasaki-Einstein manifolds.

\section{Existence of local moduli sections and spaces}
\label{existence}
In this part, we prove a general theorem on the existence of local moduli sections and spaces in the setting of \cite{modules}.

\subsection{Global setting}
\label{GS}

We follow, with slight modifications, the setting introduced in \cite{modules}.
\vspace{5pt}\\
We denote by $X$ a compact, connected, smooth manifold, by $\mathcal E$ a set of structures
on $X$. We assume that $\mathcal E$ is a Banach manifold (over $\mathbb R$ or $\mathbb C$)
and contains a closed subspace $\mathcal I$ of integrable structures (cf. Section \ref{examples} or \cite{modules} for concrete examples). 
\vspace{5pt}\\
We consider a topological group $G$ with countable topology. We assume that $G$ acts continuously on $\mathcal E$ by smooth transformations, 
preserving $\mathcal I$.
\vspace{5pt}\\
Given $J_0\in\mathcal I$, we are interested in finding a local section to the action of $G$ at $J_0$, which has moreover good properties of minimality. As explained in Section \ref{intro}, the crucial point is to define minimality in this very general context.
\vspace{5pt}\\
Before proceeding, let us precise some notations and definitions. 
\vspace{5pt}\\
Given two Banach manifolds, by a smooth map between them, we mean a $C^\infty$ one. In the case where they are modeled over a complex vector space, we assume that the differential of the map commutes with complex multiplication (hence the map has to be thought of as a holomorphic map).  A smooth bijective map with smooth inverse is an isomorphism (hence we use the generic word isomorphism instead of diffeomorphism and biholomorphism). In a Banach vector space, a map is analytic if it is analytic in the sense of
\cite{Douady}. An analytic set is the zero set of a finite number of analytic functions.
\vspace{5pt}\\
We require that a smooth submanifold of a Banach manifold has closed tangent space at each point. If $K$ is a closed subset of $\mathcal E$, then a map from $K$ to a Banach manifold is $C^k$, respectively smooth, if it is the restriction of a $C^k$, respectively smooth map defined on some open set of $\mathcal E$ containing $K$. A smooth bijective map between closed subsets of Banach manifolds whose inverse is smooth is called an isomorphism.
\vspace{5pt}\\
We denote by $J\cdot g$ the action\footnote{We make $G$ act on the right. In the many cases where $J$ is encoded by a $1$-form with values in some bundle, this means that $G$ acts by pullback.} of an element $g\in G$ onto a structure 
$J\in\mathcal E$. We say that two structures $J_1$ and $J_2$ are equivalent, and write $J_1\sim J_2$, if they belong to the same $G$-orbit. 
\subsection{Existence of a local section}
\label{local section}

Let $J_0\in\mathcal I$. We assume that there exists a Banach vector space
$T$ and a homeomorphism
\begin{equation}\label{Gcarte}
0\in V\subset T\buildrel \phi\over \longrightarrow e\in W\subset G
\end{equation}
between a connected open neighborhood $V$ of $0$ and a connected open neighborhood $W$
of the neutral element $e$ of $G$ such that the action of $G$ onto $\mathcal E$ is
smooth in this chart, that is
\begin{equation}\label{action}
(\xi,J)\in V\times U\longmapsto J\cdot\phi(\xi)\in\mathcal E
\end{equation}
is smooth for $U$ a connected open neighborhood of $J_0$.

\begin{remark}
Observe that chart (\ref{Gcarte}) depends on $J_0$. This may seem curious at first sight, but this is exactly what happens in the classical case of complex structures (cf. Section \ref{complexe}).
\end{remark}

Call $L$ the differential of (\ref{action}) at $(0,J_0)$. Let 
$$
E:=\text{Ker }\restriction{L}{T\times \{0\}}
$$
seen as a subspace of $T$. 

\begin{hypotheses}\hfill
\vspace{-5pt}\begin{description}
\item[(H1)] The vector subspace $E$ admits a closed complement $E^\perp$ in $T$.
\item[(H2)] The differential $L$ has the form $L(\xi,\omega)=\omega+P\xi$ for some linear
bounded operator $P : T\to T_{J_0}\mathcal E$.
\item[(H3)] Set $F:=\text{Im }P$. Then $F$ is closed in $T_{J_0}\mathcal E$ and admits a closed 
complement $F^\perp$.
\end{description}
\end{hypotheses}

\begin{remark}
(H1) and the second part of (H3) are automatically satisfied if $\mathcal E$ is an Hilbert manifold, which is often the case in practice 
(cf. Section \ref{complexe}). (H2) is satisfied in many interesting cases, and, in any case, it is very easy to check. So the crucial point is the first part of (H3). Checking that $\text{Im }P$ is closed is usually the hard point.
\end{remark}

\begin{remark}
$P$ will often - but not always - be a differential operator, cf. section \ref{diff}.
\end{remark}

Let $\tilde K$ be a submanifold of $U$ passing through $J_0$ and tangent to $F^\perp$ at $J_0$ (if $\exp : U'\subset T_{J_0}\mathcal E\to U$ is a local chart 
at $J_0$, just take $\exp (F^\perp\cap U')$ as $\tilde K$).
\vspace{5pt}\\
We are now in position to prove the existence of a local section.

\begin{proposition}
\label{lcexistence}
Assume {\rm (H1), (H2)} and {\rm (H3)}. Then, shrinking $V$ and $U$ if necessary, the map $\Phi$ from $E^\perp\cap V\times\tilde K$ to $U\subset\mathcal E$ defined by
\begin{equation}
\label{Kuranishi}
(\xi, J)\longmapsto \Phi(\xi, J):=J\cdot \phi(\xi)
\end{equation}
is an isomorphism at $(0, J_0)$.
\end{proposition}

\begin{proof}
The map (\ref{Kuranishi}) is smooth with differential at $(0,J_0)$ equal to
$$
(\xi,\omega)\in E^\perp\times F^\perp\longmapsto \omega\oplus P\xi\in F^\perp\oplus F.
$$
It is an isomorphism, hence the conclusion follows from the inverse function theorem.
\end{proof}

Set now $K=\tilde K\cap\mathcal I$. This is a closed subset of $\mathcal I$. It follows
from the $G$-invariance of $\mathcal I$ that the map
\eqref{Kuranishi} is an isomorphism from $E^\perp\cap V\times K$ to $U\cap\mathcal I$.
\vspace{5pt}\\
The set $K$ is the local section we were looking for. Notice that we have a smooth retraction map
\begin{equation}
\label{retraction}
J\in U\cap\mathcal I\longmapsto \Xi(J):=\left (\Phi^{-1}\right )_2(J)
\end{equation}
whose fibers are included in the $G$-orbits. Here $(\Phi^{-1})_2$ denotes the second component of $\Phi^{-1}$.

\begin{remark}
We insist on the fact that $K$ {\bf is just a closed set} and that a function on $K$ is an isomorphism means that both this function and its inverse are restriction of a smooth map between manifolds. Of course, in many cases, $K$ is a manifold or a analytic space and the isomorphism is really an isomorphism in the sense of the corresponding category, cf. Section \ref{examples}.
\end{remark}

\begin{remark}
 Here retraction just means that 
 $$
  \Xi\circ\Xi\equiv \Xi.
 $$
\end{remark}

\begin{remark}
\label{star}
In the Hilbert case, since $\text{Im }P$ is closed, a complementary subspace is given by the kernel of the Hilbert adjoint $P^*$. If $P$ is a differential operator, using the appropriate norms, the kernel of the formal adjoint gives also a complementary subspace.
\end{remark}

\subsection{Minimality conditions}
\label{minimal}

We want to prove that, under some additional assumptions, the local section has good properties
of minimality. We will consider two minimality conditions.
\vspace{5pt}\\
In the classical case of complex structures (cf. Section \ref{complexe}), versality in the sense of \cite{K-S} is the good minimality condition. However,
it cannot be adapted to our general setting for it supposes to have an associated Kodaira-Spencer theory. To be more precise, in the complex case, 
Kodaira-Spencer theory tells us that versality means bijectivity of the Kodaira-Spencer map at $J_0$. But here we do not have a well defined Kodaira-Spencer map; we do not even have a well defined notion of a deformation as
a flat morphism in some sense. Indeed, the point here is that we try to define minimality in situations where there is no associated Kodaira-Spencer map. 
\vspace{5pt}\\
As a substitute, versality in the complex case can also be defined as minimality of the dimension of the local section at the base point. But this is also unadapted to our setting since it supposes finite dimension of the local section, which is precisely an hypothesis we want to discard, since it is not satisfied in many examples.
 \vspace{5pt}\\
In \cite{modules}, we proposed a definition of a local moduli section which has to be thought of as a substitute for the notion of versal deformation space. We prove it to be equivalent to versality for complex structures. 
\vspace{5pt}\\
We deal with this notion in Section \ref{sections}. But, before that, we will see now slightly more general minimality conditions, which appear naturally in our setting. The central idea is very simple: in a minimal local section, the repetitions (that is the subset of points encoding a fixed structure up to $G$-action) should be minimal. The case of complex structures shows that we cannot prevent repetitions, even repetitions of the base point $J_0$. It also shows that it is too much to expect these repetitions to be countable subsets {\it for all} $J\in K$, but that we should ask for this at $J_0$. In other words, one should call minimal a local section in which there is no path in $K$ starting at $J_0$ encoding a trivial deformation of $J_0$. Both the minimality conditions introduced here and that of \cite{modules} are precise statements saying that. The differences lie in defining what is a trivial deformation, and are related to the validity or not of the Fischer-Grauert property (see Section \ref{sections}).  
\vspace{5pt}\\
Let us introduce the following two minimality conditions.
\begin{description}
\item[(MC1)] Let $c:[0,\epsilon)\to G$ be a continuous path starting at $e$. For $t\in [0,\epsilon)$, define
$J(t):=\Xi(J_0\cdot c(t))$.
Then the continuous path $J$ in $K$ satisfies $J\equiv J_0$.
\end{description}

In other words, if (MC1) is fulfilled, the intersection of $K$ with the local $G$-orbit of $J_0$ does not contain any non-constant continuous path.

\begin{description}
\item[(MC2)] Up to shrinking $W$, we have
$$
g\in W,\ J_0\cdot g\in K\Longrightarrow J_0\cdot g=J_0.
$$
\end{description}

In other words, if (MC2) is satisfied, the intersection of $K$ with the local $G$-orbit of $J_0$ is equal to $\{J_0\}$.
\vspace{5pt}\\
Of course, (MC2) implies (MC1); but more is true.

\begin{proposition}\label{MC}
{\rm (MC1)}$\iff${\rm (MC2)}.
\end{proposition}

\begin{proof}
Assume that (MC2) is not satisfied. Then, since $G$ has a countable topology, we can find a sequence $(g_n)$ in $G$ converging to $e$ and such that $J_0\cdot g_n$ is in $K\setminus\{J_0\}$ for all $n$.
\vspace{5pt}\\
For $n$ big enough, $g_n$ belongs to $W$ hence $g_n=\phi(\xi)$ for some $\xi$. Set
$$
c(t)=\phi(t \xi)\quad\text{ and }\quad J(t)=\Xi(J_0\cdot c(t)).
$$
This is a continuous path in $K$ and in the local orbit of $J_0$. Besides, it is non-constant since $J(1)=J_0\cdot g_n$ is different from $J_0$. Hence (MC1) is not satisfied.
\end{proof}

Because of Proposition \ref{MC}, we will from now on refer to both (MC1) and (MC2) as the shortened (MC). 

\begin{definition}
If a local section fulfills {\rm (MC)}, we say it is {\it minimal}.
\end{definition}

We now need to define some more assumptions.

\begin{hypotheses}\hfill
\vspace{-5pt}\begin{description}
\item[(H4)] The isotropy group
$$
G_{J_0}:=\{g\in G\mid J_0\cdot g=J_0\}
$$
is a local Banach submanifold at $J_0$, that is there exists a smooth map
$$
E\cap V\buildrel\psi\over \longrightarrow G_{J_0}\cap W.
$$
\item[(H5)] The map
$$
(g,h)\longmapsto \mu(g,h)=g\circ h\in W
$$
for $g\in G_{J_0}$ and $h\in G$, is $C^1$ (that is $C^1$ in the chart (\ref{Gcarte})), with differential at $(e,e)$ equal to
$$
(\xi,\eta)\in E\times T\longmapsto \xi+\eta\in T.
$$
\end{description}
\end{hypotheses} 

\begin{remark}
We emphasize that (H5) means that the composition of elements $g$ and $h$ of $G$ is $C^1$, {\it only when $g$ belongs to the isotropy group of $J_0$}. In many cases, for example when $G$ is the group of diffeomorphisms of class $C^k$ or of Sobolev class $W^l$, the composition map in the group is only continuous (see the proof of Proposition \ref{Pelliptic}). So for (H5) to be fulfilled, we need that $G_{J_0}$ contains only diffeomorphisms of higher regularity, as it is the case in both examples of Section \ref{examples}.
\end{remark}

We have now

\begin{proposition}
\label{MCproof}
Assume {\rm (H1), (H2), (H3), (H4)} and {\rm (H5)}. Then the local section $K$ of Proposition \ref{lcexistence} is minimal.
\end{proposition}

\begin{proof}
We shall prove (MC2). As permitted by assumption (H4), let $\psi$ be a local chart of $G_{J_0}$ at $e$. We assume that its differential at $0$ (written in the chart (\ref{Gcarte})) is the identity. Then, shrinking $V$ and $W$ if necessary, we may assume that the map
\begin{equation}
\label{multi}
(\xi,\chi)\in (E\cap V)\times (E^\perp\cap V)\longmapsto \mu (\psi(\xi),\phi(\chi))\in W
\end{equation}
is an isomorphism, since it is a smooth map by (H5) and since its differential at $(0,0)$ is the identity.
\vspace{5pt}\\
Let $g\in W$ be such that $J_0\cdot g\in K$. Then $g$ can be written as $\psi(\xi)\circ\phi(\chi)$ and we have
$$
J_0\cdot g=J_0\cdot (\psi(\xi)\circ\phi(\chi))=J_0\cdot\phi(\chi)
$$
the last equality coming from the fact that $\psi(\xi)$ belongs to the isotropy group of $J_0$.
\vspace{5pt}\\
But 
$$
\Xi(J_0\cdot\phi(\chi))=J_0=\Xi(J_0\cdot g)=J_0\cdot g
$$
the last equality coming from the fact that $J_0\cdot g$ is supposed to be in $K$. Hence
$J_0\cdot g=J_0$ and (MC2) is verified.
\end{proof}

\subsection{Local moduli sections}
\label{sections}

In \cite{modules}, we proposed a definition of a local moduli section. It says that $K$ is a local moduli section if there exists an smooth retraction from a neighborhood of $J_0$ in $\mathcal I$ onto $K$ (condition A1) and if any smooth path $J : [0,\epsilon)\to K$ starting at $J_0$ and all of whose points are equivalent to $J_0$ is indeed constant (condition A2).
\vspace{5pt}\\
We want to compare this definition with that of a minimal local section. First, observe that condition A1 is always fulfilled, even for a general local section, because of the existence of (\ref{retraction}). Then, notice that smoothness is not a problem, since replacing continuous by smooth in (MC1) yields an equivalent condition. This is due to the fact that a continuous path in a Banach space can be approximated by smooth ones.
\vspace{5pt}\\
Also, condition A2 obviously implies (MC1). However, the converse is only true if we have:

\begin{description}
\item[(FG) property] Given 
$J : t\in [0,\epsilon)\to J_t\in K$ 
a smooth path of structures with $J(0) = J_0$ and all equivalent to $J_0$, there exists a smooth path 
$c : t\in [0,\epsilon)\to c(t)\in G$ with $c(0) = e\in G$ and
such that $J(t)=J_0\cdot c(t)$ for all $t$.
\end{description}

Of course, (FG) stands for Fischer-Grauert. Define a smooth deformation of $J_0$ as
a smooth map from a smooth base manifold $B$ to $K$ sending a base point $0$ onto $J_0$, and a trivial deformation of $J_0$ as a smooth deformation $J : b\in B\mapsto J_0\cdot c(b)$ for $c$ a smooth map from $B$ to $G$. Then, these definitions are consistent with those used for complex structures (cf. \cite{circle}) and (FG) property can be rephrased as: a smooth deformation of $J_0$ all of whose points are equivalent to $J_0$ is trivial. This is exactly the (smooth version) of Fischer-Grauert's Theorem (see \cite{PS}).
\vspace{5pt}\\
To sum up, we have

\begin{proposition}
\label{MC+FG}
A minimal local section is a local moduli section if and only if {\rm (FG)} property is true.
\end{proposition}

To rephrase this proposition, minimality states that there is no non-constant path in $K$ encoding a trivial deformation of $J_0$, whereas condition A2 means that there is no non-constant path all of whose points are equivalent to $J_0$.
\vspace{5pt}\\
In a setting where (FG) property is true (as in the case of complex structures), both definitions are the same. However, in a setting where (FG) property is not true, then the good definition to take is that of minimality, the philosophy being that a path all whose points encode the same structure but which is not a trivial deformation encodes important information and cannot be removed from the local section. Indeed, if (FG) property is not true, there cannot exist a local moduli section.
\vspace{5pt}\\
We may now state
\begin{definition}
 \label{Ktype}
 A minimal local section at $J_0$ is also called {\it a Kuranishi type space} of $J_0$.
\end{definition}

Proposition \ref{complexe} should justify this terminology.

\begin{remark}
 To be more precise, a local moduli section as defined in \cite{modules} is supposed to be an analytic space in the sense of \cite{Douady}. In this section, we drop this requirement. However, notice that the analytic structure is important to prove \cite[Proposition 4.1]{modules}, see Section \ref{examples} and remark \ref{Douadyreduced}.
\end{remark}

\subsection{Local moduli space}
\label{localmoduli}

To finish with this part, we want to compare the previous definitions with the more classical one of local moduli space. Recall the

\begin{definition}
\label{lms}
A local section $K$ is {\it a local moduli space} if there exists an open neighborhood $W$ of $e$ in $G$ such that
$$
J\in K,\ g\in W,\ J\cdot g\in K\Longrightarrow J\cdot g=J.
$$
\end{definition}

This is stronger than the notions of minimal local section and local moduli section. Indeed,

\begin{proposition}
\label{localmoduli>section}
Assume that $K$ is a local moduli space. Then $K$ is a local moduli section at any point $J_0\in K$.
\end{proposition}

\begin{remark}
\label{universel}
If we compare with the vocabulary of Kodaira-Spencer theory, a local moduli space is a universal space, whereas a minimal local section has to be thought of as a versal space
(cf. Proposition \ref{complexe} and the discussion in subsection \ref{sections}).
\end{remark}

\begin{proof}
Take $J_0$ in $K$. Obviously, $K$ is minimal, since (MC2) condition is nothing else than the condition in definition \ref{lms} applied to $J_0$. 
\vspace{5pt}\\
Moreover, let 
$$
	J : t\in [0,\epsilon ) \longrightarrow J_t\in K
$$
be a smooth path all of whose points are equivalent to $J_0$ and with $J(0) = J_0$. For every $t$, we may choose an element $g_t$ of G such that $J(t)=J_0\cdot g_t$. Assume $J$ is non-constant. Since it is smooth, it is locally injective at each point of an open subset $I$ of $[0,\epsilon )$. Moreover, we may assume that $J$ is not constant on any interval containing $0$ without any loss of generality. Hence, we have that $0$ belongs to the closure of $I$. Since $I$ is uncountable whereas $G$ has a countable topology, the uncountable family $((g_t)_{t\in I})$ of $G$ must have an accumulation point. The same trick shows that there exists such an accumulation point arbitrarily close to $0$ that can be assumed to belong to $I$. All that says that we can find $t_\infty\in I$ arbitrarily close to $0$ and $t_n\to t_\infty$ with all $t_n$ different from $t_\infty$ and with $g_{t_n}$ converging to $g_{t_\infty}$.
\vspace{3pt}\\
For $n$ big enough, $(g_{t_\infty})^{-1}\circ g_{t_n}$ belongs to $W$ and 
$$
J(t_\infty)\cdot \left ((g_{t_\infty})^{-1}\circ g_{t_n}\right )=J_0\cdot g_{t_n}=J(t_n)\in K
$$
so we have $J(t_n)=J(t_\infty)$, contradicting local injectivity at $t_\infty$.
\end{proof}

Let us give additional hypotheses under which we have a local moduli space. 

\begin{hypotheses}\hfill
\vspace{-5pt}\begin{description}
\item[(H4')] For all $J\in K$, its isotropy group $G_J$ is a local Banach submanifold with tangent space at identity $E_J$ isomorphic to $E$.
Moreover, there exists a smooth map $\psi : U\times (E\cap V)\to G$ such that, for all $J\in K$, 
$$
\psi (J,E\cap V)=G_J\cap W
$$
and
$$
\psi(J,0)\equiv e.
$$
\item[(H5')] 
The map 
$$
(J,\xi,h)\in\tilde K\times E\cap V\times W \longmapsto \psi(J,\xi)\circ h\in G
$$
is $C^1$ (that is $C^1$ in the chart (\ref{Gcarte})), and the differential at $(0,e)$ of
$$
(\xi,h)\longmapsto \psi(J_0,\xi)\circ h
$$
is equal to
$$
(\xi,\eta)\in E\times T\longmapsto \xi+\eta\in T.
$$
\end{description}
\end{hypotheses}

\begin{remark}
Set 
$$
\mathcal C=\{(J,\xi)\in K\times V\mid \xi\in E_J\}.
$$
Then the natural projection map $\mathcal C\to K$ satisfies
$$
\begin{CD}
\mathcal C@>\psi^{-1}>> K\times (E\cap V)\\
@VVV @VV\text{1st projection}V\\
K@=K
\end{CD}
$$
up to shrinking $K$ and $V$. If $\mathcal C$ and $K$ are Banach $\mathbb C$-analytic spaces, then this is exactly saying that $\mathcal C\to K$ is a smooth morphism (cf. \cite[p.28]{Douady}).
\end{remark}

\begin{remark}
(H4'), respectively (H5'), 
is stronger than (H4), respectively (H5), 
 and is meant to replace it.
\end{remark}

Then, we have

\begin{proposition}
\label{LMS}
Assume {\rm (H1), (H2), (H3), (H4')} and {\rm (H5')}. Then $K$ is a local moduli space.
\end{proposition}

\begin{proof}
This is close to the proof of proposition \ref{MCproof}. We need a sort of uniform version of it. 
Let
\begin{equation}
\label{uniform}
 (J,\xi,\chi)\in\tilde K\times (V\cap (E\times E^\perp))\longmapsto (J,\psi(J,\xi)\circ\phi(\chi))\in \tilde K\times G.
\end{equation}
By (H4') and (H5'), this is a smooth map. 
By (H5'), its differential at $(J_0,0,0)$ is equal to
$$
(\omega, \xi,\chi)\in T_{J_0}\tilde K\times E\times E^\perp\longmapsto (\omega,\xi\oplus\chi)\in T_{J_0}K\times T
$$
so is an isomorphism. There is no $\omega$ term in the second component because of the condition $\psi(J,0)\equiv e$. Up to shrinking $\tilde K$, $V$ and $W$, we may assume that (\ref{uniform}) is an isomorphism onto $\tilde K\times W$. 
Then, take
$g\in W$ and $J\in K$. We may write $g$ as $\psi(J,\xi)\circ \phi(\chi)$ using (\ref{uniform}). And we have
$$
J\cdot \left (\psi(J,\xi)\circ \phi(\chi)\right )=J\cdot\phi(\chi)
$$
since $\psi(J,\xi)$ belongs to $G_J$.
\vspace{5pt}\\
Assume now that $J\cdot g\in K$, then, since $J\in K$ and $g\in W$, by (\ref{Kuranishi}) and (\ref{retraction}), we have
$$
\phi(\chi)=e\quad\text{ and }\quad J\cdot g=J.
$$
as claimed.
\end{proof}

Observe that conditions (H4'), (H5') are satisfied in the following case.

\begin{corollary}
\label{H0=0}
Assume that for all $J\in K$, we have $G_J\cap W$ reduced to $e$. Then $K$ is a local moduli space.
\end{corollary}

\begin{proof}
Just take $\psi$ as the constant mapping onto $e$. (H4') and (H5') are trivially satisfied.
\end{proof}
 
Let us sum up all the previous result in the following Theorem.

\begin{theorem}
\label{main}
Let $X$ be a compact smooth manifold. Consider $\mathcal E$, $\mathcal I$ and $G$ as in Section {\rm \ref{GS}}. Let $J_0\in\mathcal I$. Assume {\rm (H1), (H2)} and {\rm (H3)}.
\begin{enumerate}
\item There exists a local section $K$ to the $G$-action and {\rm (\ref{Kuranishi})} is an isomorphism. 
\item Assume {\rm (H4)} and {\rm (H5)}. Then $K$ is a Kuranishi type space. And it is a local moduli section if and only if {\rm (FG)} property is true.
\item Assume {\rm (H4')} and {\rm (H5')}. Then $K$ is a local moduli space.
\end{enumerate}
\end{theorem}

\subsection{Local moduli space and quotient stacks}
\label{stacks}

In the classical case of complex structures, one can get rid of the notion of versality by using stacks. More precisely, given $X_0$ a compact complex manifold and $\text{Kur}$ its germ of Kuranishi space, then each $1$-parameter subgroup of $\text{Aut}^0(X)$ acts on $\text{Kur}$ and the quotient stack $[\text{Kur}/\text{Aut}^0(X)]$ is universal (cf. remark \ref{universel}). \vspace{5pt}\\
In our general setting, assuming the existence of a Kuranishi type space $K$ for a structure $J_0$, each $1$-parameter subgroup of its isotropy group acts on the germ of $K$ at $J_0$ via the formula
$$
(g,J)\in G_{J_0}\times K\longmapsto \Xi (J\cdot g)\in K.
$$
We can form the quotient stack $[K/G_{J_0}]$. Its is natural to call it a {\it local moduli stack} since the associated topological quotient satisfies the requirement of Definition \ref{lms}. One of the interest of this approach is that the functorial description of this stack should provide the good notion of a flat family of structures. Another interest is that it is the first step in the concrete description of a Teichm\"uller space as a stack, as it is done in \cite{Teich} for complex structures.\vspace{5pt}\\
Giving a precise description of such a quotient stack should be possible at least in the case where $K$ is an analytic space and $G_{J_0}$ is a complex Lie group.

\subsection{Rigidity}
From Theorem \ref{main}, we may easily deduce a rigidity result in the spirit of that saying that a compact complex manifold $X$ with first cohomology group with values in the sheaf of holomorphic vector fields $H^1(X,\Theta)$ being zero is rigid.
\vspace{5pt}\\
As in Section \ref{sections}, define a smooth (or holomorphic) deformation of $J_0$ as
a smooth/holomorphic map from a smooth/holomorphic base manifold $B$ to $K$ sending a base point $0$ onto $J_0$; and a trivial deformation of $J_0$ as a deformation $J : b\in B\mapsto J_0\cdot c(b)$ for $c$ a map from $B$ to $G$.
\vspace{5pt}\\
Then, rigidity is defined classically as follows.
\begin{definition}
The structure $J_0$ is called {\it rigid} if every germ of deformation of $J_0$ is isomorphic to a trivial deformation.
\end{definition}

\begin{corollary}
Assume {\rm (H1), (H2)} and {\rm (H3)}. Assume that $K$ is reduced to a point. Then $J_0$ is rigid.
\end{corollary} 

\begin{remark}
\label{remarkZ}
Assume that $\mathcal E$ and $G$ are local Hilbert manifolds. Assume that $\mathcal I$ is given locally as the zero set of some analytic map $Q$. Denote by $Q^{lin}$ the linear part of $Q$. Then, consider the vector space
\begin{equation}
\label{Zariski}
\{ J\in T_{J_0}\mathcal E\mid P^*J=Q^{lin}J=0\}.
\end{equation}
This is the "tangent space" of $K$ at $J_0$, in the sense that the derivative at $0$ of any smooth map $c$ into $K$ with $c(0)=J_0$ lies in it. Assume that $K$ is a closed subset of a submanifold of $\mathcal E$ whose tangent space is \eqref{Zariski} (this is obviously the case if $K$ itself is a submanifold of $\mathcal E$ or an analytic subspace of $\mathcal E$ in any reasonable sense). We have
\begin{corollary}
In the setting of remark {\rm \ref{remarkZ}}, assume {\rm (H1), (H2)} and {\rm (H3)}. Assume also that {\rm \eqref{Zariski}} is reduced to a point. Then $J_0$ is rigid.
\end{corollary} 

\end{remark}

\section{Structures given by the sections of a bundle}
\label{existencesmooth}

\subsection{Setting}
\label{newsetting}

In many cases, the space $\mathcal E$ is a space of sections of a bundle. To be more precise, we assume now the following conditions.

\begin{itemize}
\item The set $\mathcal E$ is a subset of the space of $C^\infty$ sections of a fiber bundle over $X$ and a Fr\'echet manifold. Moreover,  $T_{J_0}\mathcal E$ is the space of $C^\infty$ sections of a vector bundle over $X$.

\item The group $G$ is a subgroup of the group of $C^\infty$ diffeomorphisms of $X$, and the vector space $T$ is a subspace of the Lie algebra of vector fields on $X$. Moreover, \eqref{Gcarte} is a Fr\'echet chart for $G$ at $e$.
\end{itemize}

Denote by $B_0$ and $B_1$ the two smooth vector bundles  over $X$ such that
$$
T=\Gamma^\infty(B_0)\qquad\text{and}\qquad T_{J_0}\mathcal E=\Gamma^\infty(B_1)
$$
where $\Gamma^\infty$ denotes the set of smooth sections of a bundle.
\vspace{5pt}\\
In this situation, we can be more precise than in Theorem \ref{main}. However, there is an additional problem that appears, that of the regularity of the sections. To run Theorem \ref{main}, we need to use a Sobolev completion and work with $W^l$ sections with $l>1+\dim X/2$; that is, we will suppose that the operator $P$ extends to appropriate Sobolev $l'$ and $l$-completions of $T$ and $T_{J_0}\mathcal E$ respectively. 
We also assume that the Fr\'echet structure of $\mathcal E$ extends as an Hilbert structure on $\mathcal E^l$, the corresponding subset of $W^l$ sections. In other words, given a Fr\'echet chart of $\mathcal E$ at some point $J$ modeled on $\Gamma^\infty(B_1)$, we may assume that it extends as a Hilbert chart of $\mathcal E^l$, modeled on the completion $\Gamma^l(B_1)$ of $\Gamma^\infty(B_1)$. In the same way, we assume that the Fr\'echet chart \eqref{Gcarte} from $T$ to $G$ extends as a Hilbert map from the $l'$-completion $T^{l'}$ to the $l'$-completion $G^{l'}$.
In principle $l'$ can be different from $l$. For instance, if $P$ is a differential operator of degree $k$, then $l'=l+k$.\vspace{5pt}\\
Then, fixing such an $l$, we are in the Hilbert setting. 
 
\subsection{The case of a differential operator}
\label{diff}

In this subsection, we replace the hypotheses (H1), (H2) and (H3) with the following ones.

\begin{hypotheses}\hfill
\vspace{-5pt}\begin{description}
\item[(H2diff)] The differential $L$ has the form $L(\xi,\omega)=\omega+P\xi$ for some differential operator $P : T\to T_{J_0}\mathcal E$.
\item[(H3diff)] The differential operator $P$ is elliptic with $C^\infty$ coefficients.
\end{description} 
\end{hypotheses}

\begin{remark}
Here, by elliptic, we mean that $P$ has an injective symbol, not a bijective one. In this latest case, we speak of a strongly elliptic operator.
\end{remark}

Also we add a $l$ to an hypothesis to say that it is valid for the particular Sobolev class $W^l$. For example, (H4$l$) means that (H4) is valid for the particular Sobolev class $W^l$, whereas (H4) means it is valid for all Sobolev classes (always assuming implicitly that $l$ is big enough).
We have now
\begin{proposition}
\label{Pelliptic}
Assume {\rm (H2diff)} and {\rm (H3diff)}. Then there exists a local section $K^l$ for all $l$. Moreover, assuming {\rm (H4$l$)} for a particular choice of $l$, then the corresponding $K^l$ is a Kuranishi type space.
\end{proposition}

\begin{proof}
Fix some $l>1+\dim X/2$. Since our vector spaces are Hilbert spaces, conditions (H1) and (H3), second part, are automatically satisfied. Moreover, following \cite[\S 3.9]{Narasimhan}, since $P$ is an elliptic operator by (H3diff), its image is closed in any Sobolev class and (H3) is completely fulfilled. Hence, by Theorem \ref{main}, there exists a local section $K^l$.
\vspace{5pt}\\
Consider the composition map
\begin{equation}
\label{cm}
(f,g)\in G^{l+k}\times G^{l+k}\longmapsto f\circ g \in G^{l+k}.
\end{equation}
Using the above assumption that $G^{l+k}$ is a subgroup of the $(l+k)$-diffeomor\-phi\-sm group of $X$, then we see that its (formal) differential at a point $(f,g)$ takes the form
$$
(\xi',\chi')\longmapsto \xi'\circ g+df(\chi')
$$
(cf. \cite[Example 4.4.5]{Hamilton}). Because of the term $df$, it does not map $W^{l+k}$-vector fields onto $W^{l+k}$-vector fields. So \eqref{cm} is not even a $C^1$ map. But it is a smooth map when we take $f$ in a submanifold of $G^{l+k}$ that contains only $C^\infty$ points.\vspace{5pt}\\
Now, still by (H3diff), $E$ is the kernel of an elliptic operator with smooth coefficients, hence contains only $C^\infty$ elements \cite[\S 3.7]{Narasimhan}. By (H4) and our assumptions on the chart \eqref{Gcarte}, the same is true for $G_{J_0}\cap W$. From this, it follows that (H5$l$) is fulfilled. We conclude by Theorem \ref{main}. 
\end{proof}
\begin{remark}
\label{important}
It is crucial to emphasize that, in the setting we use here, the fact that $P$ is elliptic {\bf does not imply} that the local section $K^l$ is finite-dimensional. This is because the tangent space to $K^l$ at the base point {\bf is not} given by the kernel of the laplacian associated to $P$ as in the classical case; but by the kernel of \eqref{Zariski} which may be completely different from this laplacian. Indeed, perhaps the main idea of our construction is to separate the integrability condition from the existence of a local section, so that the operator encoding integrability (that is the linear part of the integrability equation) is not supposed to be the same as the operator encoding the orbit (that is $P$). This allows us to gain flexibility and treat cases that cannot be treated in the classical setting.
\end{remark}

In this framework, one is faced to a delicate analytic problem. If (H2diff), (H3diff), and (H4) are satisfied, then
there exists a Kuranishi type space $K^l$ for any $l>1+\dim X/2$. Nevertheless, the interesting geometric situation is the $C^\infty$ class, for which we do not have
a Kuranishi type space, for we cannot use 
Theorem \ref{main}, the spaces of sections not being Banach spaces.
\vspace{5pt}\\
Since our deformation problems arise mainly from geometric situations, one may expect that Theorem \ref{main} is still valid in the $C^\infty$ case, 
taking as Kuranishi type space the set of $C^\infty$ points of $K^l$ (which should be the same for all $l$).
\vspace{5pt}\\
However, this is not evident at all from the point of view of differential operator theory. Indeed, in the general case (that is if $P$ is not elliptic but has closed image in each class $W^l$), we do not even know if the set of $C^\infty$ points of $K^l$ is not empty.
\vspace{5pt}\\
The only case where this problem easily disappears is the case where the tangent space of $K^l$ given in \eqref{Zariski} is the kernel of an elliptic 
operator with $C^\infty$ coefficients. Then, $K^l$ contains only $C^\infty$ structures (see \cite{Narasimhan}, \S 3.7), hence $K^l$ does not depend on $l$ and is a Kuranishi type space for $C^\infty$-structures. However, this forces this space to be finite-dimensional, which is not the case in many geometric problems, as 
in subsection \ref{riemann} and section \ref{sasakian}, and which is not the case in our setting, cf. remark \ref{important}.
\vspace{5pt}\\
If $P$ is elliptic, we have the following weaker result.

\begin{proposition}
\label{density}
Assume {\rm (H2diff)} and {\rm (H3diff)}. 
Then,
$K^l$ is equal to the $l$-completion of $K^\infty$, the space of elements of $K^l$
of class $C^\infty$.
\end{proposition}

\begin{proof}
We owe this proof to J.A. \'Alvarez-L\'opez.
\vspace{5pt}\\
Because of \eqref{Zariski}, it is enough to prove that the kernel of $P^*$ applied to $W^l$ points is the $l$-completion of the kernel of $P^*$ applied to smooth points.
\vspace{5pt}\\
First assume that $P$ is the first morphism of an elliptic complex $(E_i,P_i)$ (that is, $T=E_0$, $T_{J_0}\mathcal E=E_1$ and $P=P_0$). Hodge decomposition Theorem implies
$$
\text{Ker } P_0^* = \text{Ker }(P_1+P_0^*) \oplus P_1^*(\Gamma^\infty(E_2))
$$
in $\Gamma^\infty(E_1)$, and
$$
\text{Ker } P_0^* = \text{Ker }(P_1+P_0^*) \oplus P_1^*(\Gamma^{l+k}(E_2))
$$
in $\Gamma^l(E_1)$. Indeed,  $P_1^*(\Gamma^\infty(E_2))$ is dense in $P_1^*(\Gamma^{l+k}(E_2))$ because $\Gamma^\infty(E_2)$ is dense in $\Gamma^{l+k}(E_2)$ and $P_1^*: \Gamma^{l+k}(E_2) \to \Gamma^l(E_1)$ is continuous. Besides, $\text{Ker }(P_1+P_0^*)$ is the same in both decompositions because it contains only sections of class $C^\infty$ and has finite dimension. So $\text{Ker } P_0^*$ in $\Gamma^\infty(E_1)$ is dense in $\text{Ker } P_0^*$ in each $\Gamma^l(E_1)$.
\vspace{5pt}\\
To deal with the general case, let
$\sigma_0(x,\xi):(E_0)_x\to (E_1)_x$ be the injective symbol of $P$ where $x\in X$ and $0\ne\xi\in T_xX^*$. Let
$$
E_2 = ((TX^*)^{\otimes k} \otimes E_1)/I^2 ,
$$
where $I^2$ is the vector bundle whose fiber at $x$ is generated by vectors
$$
\xi^{\otimes k} \otimes \sigma_0(x,\xi)(v)
$$
for $\xi\in T_xX^*$ and $v\in E^0_x$. This $I^2$ is a subbundle because $\sigma_0(x,\xi)$ is injective. Also it depends differentiably in $(x,\xi)$.  
For $x\in X$ and $\xi\in T_xX^*$, define $\sigma_1(x,\xi):(E_1)_x\to (E_2)_x$ as
$$
\sigma_1(x,\xi)(v)=[\xi^{\otimes k} \otimes v] ,
$$
where the brackets denote the class modulo $I^2_x$. Then $\sigma_1(x,\xi)(v)$ is linear in $v$, differentiable in $(x,\xi)$, and homogeneous of order $k$ in $\xi$. Hence it is the principal symbol of some pseudodifferential operator $P_1:\Gamma^\infty(E_1)\to \Gamma^\infty(E_2)$ of order $k$, since pseudodifferential operators are locally defined by their symbols. 
Moreover, the sequence
$$
0 \to (E_0)_x \overset{\sigma_0(x,\xi)}{\longrightarrow} (E_1)_x
\overset{\sigma_1(x,\xi)}{\longrightarrow} (E_2)_x
$$
is exact if $\xi\ne0$.

By induction, we construct vector bundles $E_i$ and operators 
$$
P_i:\Gamma^\infty(E_i)\to \Gamma^\infty(E_{i+1})
$$ 
with symbol $\sigma_i$ for all $i$. The sequence of symbols $\sigma_i$ is exact (although infinite). 
From the properties fulfilled by $\sigma_0$ and $\sigma_1$ we have that
$$
\sigma_0(x,\xi) \sigma_0(x,\xi)^* + \sigma_1(x,\xi)^* \sigma_1(x,\xi)
$$
is an isomorphism if $\xi\ne0$. Hence $S = P_0 P_0^* + P_1^* P_1$ is an elliptic selfadjoint operator, yielding a Hodge decomposition 
$$
\Gamma^\infty(E_1)= \text{Ker } S \oplus \text{Im } S
$$
It can be refined as
$$
\Gamma^\infty(E_1)
=(\text{Ker } P_0^* \cap \text{Ker } P_1) \oplus \text{Im } P_0 \oplus \text{Im } P_1^* .
$$
from which it follows that
$$
\text{Ker } P_0^* = (\text{Ker } P_0^* \cap \text{Ker } P_1) \oplus Q
$$
with $Q \cong \text{Im } P_1^*$ given by canonical projection. 
\end{proof}

\begin{corollary}
Fix $l$. Then the set of $C^\infty$ points of $K^l$ is dense in $K^l$ and does not depend on $l$.
\end{corollary}

\subsection{Smooth versus Sobolev Kuranishi type spaces}
\label{smoothsobolev}

Proposition \ref{density} does not provide us with a Kuranishi type space for $C^\infty$ structures. To obtain such a result, we will use, if it exists, an affine connection on $\mathcal E$. The following is directly inspired in \cite{Ebin}. Assume that 
\begin{hypotheses}\hfill
	\vspace{-5pt}\begin{description}
		\item[(H6)] For all $\xi\in V$, the structure $J_0\cdot \phi(\xi)$ is of class $C^\infty$ if and only if $\xi$ is
		of class $C^\infty$.
		\item[(H7$l$)] There exists a smooth affine connection on $\mathcal E^l$ which is invariant under the action of $G^{l+k}$.
		\item[(H8)] The vector bundle $B_1$ is a natural bundle and action \eqref{action} is the natural action.
	\end{description}
\end{hypotheses}

In particular, (H6) implies that $J_0$ itself is assumed to be of class $C^\infty$. We also observe that, if $J$ can be identified with some differential form $\omega$ with values in a vector bundle, then (H8) implies that the action is given by pull-back, i.e.  $\omega\cdot \phi(\xi)=\phi(\xi)^*\omega$.
\vspace{5pt}\\
Then, we can associate to the affine connection an exponential map \cite{Lang}, \S IV.4 and \S VII.7
$$
\text{Exp }: T\mathcal E^l\longrightarrow \mathcal E^l.
$$
Because of (H7$l$), it is invariant under the action of  $G^{l+k}$. Consider the local orbit $\mathcal O$ of $G$ at $J_0$. Assuming (H1$l$), (H2$l$) and (H3$l$) 
then, by \eqref{Kuranishi}, it is closed with tangent space
equal to $F=\text{Im }P$ and orthogonal complement $F^\perp$. Then we have

\begin{lemma}
	\label{exp}
	The map
	\begin{equation}
		\label{expN}
		(\xi,\omega)\in E^\perp\times F^\perp\longmapsto \exp_N(\xi,\omega):=\text{\rm Exp }((J_0,\omega)\cdot \phi(\xi))
	\end{equation}
	is a local isomorphism at $(0, 0)$ onto the neighborhood of $J_0$ in $\mathcal E^l$.
\end{lemma}

\begin{proof}
	This is a standard fact about the exponential map that the composition map
	$$
	\mathcal E^l\longrightarrow T\mathcal E^l\longrightarrow \mathcal E^l
	$$
	of the projection map with the inclusion map as zero section is the identity. From this, the differential of \eqref{expN} at $(0,0)$ is given by
	$$
	(\xi,\omega)\in E^\perp\times F^\perp\longmapsto \omega+P\xi\in T_{J_0}\mathcal E
	$$
	and this is now a direct application of the inverse function theorem.
\end{proof}
Then, calling $\pi$ the projection $E^\perp\times F^\perp$ onto the first factor, we have

\begin{proposition}
	\label{expCinfini}
	Assume {\rm (H2diff)} and {\rm (H3diff)}, or {\rm (H1$l$)}, {\rm (H2$l$)} and {\rm (H3$l$)}. Assume also {\rm (H7$l$)}. Then we may define the local section of {\rm \eqref{Kuranishi}} as 
	\begin{equation}
		\label{Ktilde}
		\tilde K:=\{\text{\rm Exp } (J_0,\omega)=\exp_N(0,\omega) \mid \omega\in F^\perp\ \text{\rm close to }0\}
	\end{equation} 
	Moreover the inverse map of
	{\rm \eqref{Kuranishi}} is given by
	\begin{equation}
		\label{inverse}
		J\longmapsto \Phi^{-1}(J):=(\pi((\exp_N)^{-1}(J)), J\cdot (\phi(\pi((\exp_N)^{-1}(J)))^{-1})).
	\end{equation}
	
\end{proposition}

\begin{proof}
	This is a straightforward computation using Lemma \ref{exp}. We have
	$$
	\begin{aligned}
	\Phi^{-1}\circ\Phi (\xi, J)&=\Phi^{-1}\circ \Phi(\xi,\text{Exp } (J_0,\omega))\cr
	&=\Phi^{-1}(\text{Exp }(J_0,\omega)\cdot\phi(\xi)) \cr
	&=\Phi^{-1} (\text{Exp }((J_0,\omega)\cdot\phi(\xi))\cr
	&=\Phi^{-1}(\exp_N(\xi,\omega))
	\end{aligned}
	$$
	using the invariance of the exponential map. Hence,
	$$
	\pi((\exp_N)^{-1}(\exp_N(\xi,\omega))=\xi
	$$
	and, using \eqref{expN} and remembering that $J$ is $\text{Exp } (J_0,\omega)$,
	$$
	\Phi^{-1}\circ\Phi (\xi, J)=(\xi,J\cdot (\phi(\xi))^{-1}\cdot\phi(\xi))=(\xi,J).
	$$\end{proof}

As in Section \ref{diff}, define $K^\infty$ as the set of $C^\infty$ points\footnote{However, here, $l$ is fixed and, strictly speaking, $K^\infty$ depends on $l$.} of $K^l$.
We have now

\begin{proposition}
	\label{Cinfini}
	Assume {\rm (H2diff)} and {\rm (H3diff)}, or {\rm (H1$l$)}, {\rm (H2$l$)} and {\rm (H3$l$)}. Assume also {\rm (H6)}, {\rm (H7$l$)} and {\rm (H8)}.
	Define $\tilde K$ as in {\rm \eqref{Ktilde}}. Then both maps {\rm \eqref{Kuranishi}} and {\rm \eqref{inverse}} preserve the $C^\infty$ class. In particular, the map 
	{\rm \eqref{Kuranishi}} is an 
	isomorphism for $C^\infty$ structures and $K^\infty$ is a local section, and a Kuranishi type space if $K^l$ is.
\end{proposition}

\begin{proof}
	By (H8), the exponential map is invariant by pull-back by diffeomorphisms, hence commutes with Lie derivatives, so is a homeomorphism from the set of smooth points of $T\mathcal E$ to the set of smooth points of $\mathcal E$ (cf. \cite[Theorem 7.5]{Ebin}). This, together with (H6) implies that the map \eqref{expN} is also a homeomorphism from the set of smooth points of $E^\perp\times F^\perp$ onto that of $\mathcal E$. Then, we deduce that both formulas \eqref{Kuranishi} and \eqref{inverse} preserve the $C^\infty$ class.
\end{proof}

We collect all the previous results in the following statement.

\begin{theorem}
	\label{mainsmooth}
	Let $X$ be a compact smooth manifold. Consider $\mathcal E$, $\mathcal I$ and $G$ as in Section {\rm \ref{newsetting}}. Let $J_0\in \mathcal I$. Assume {\rm (H1$l$), (H2$l$), (H3$l$)} for some $l$, or {\rm (H2diff), (H3diff)}. Assume also {\rm (H6), (H7$l$)} and {\rm (H8)}.
	\begin{enumerate}
		\item  There exists a local section $K^\infty$ to the $G$-action for $C^\infty$ structures and {\rm \eqref{Kuranishi}} is an isomorphism.
		\item Assume {\rm (H4$l$)} and {\rm (H5$l$)}. Then $K^\infty$ is a Kuranishi type space for $C^\infty$ structures. And it is a local moduli section if and only if {\rm (FG)} property is true.
		\item Assume {\rm (H4'$l$)} and {\rm (H5'$l$)}. Then $K^\infty$ is a local moduli space for $C^\infty$ structures.
	\end{enumerate} 
\end{theorem}

\section{Three classical examples}
\label{examples}

In this part, we run the previous definitions and propositions in three classical cases: complex structures, riemannian metrics and ASD connections.

\subsection{Complex structures}
\label{complexe}

This is the foundational example, which inspires in all the definitions we gave. Here $\mathcal E$ is the set of almost complex operators of, say, class $W^{l}$, and $\mathcal I$ is the subspace of integrable ones, hence of complex structures. Then $G$ is the group of diffeomorphisms
of class $W^{l+1}$. It is a classical fact that $\mathcal E$ is a Hilbert manifold over $\mathbb C$, and $\mathcal I$ a closed subset. It is even an analytic subspace in the sense of \cite{Douady}. Let $J_0$ be a complex structure of class $C^\infty$. Then $T$ is the
Hilbert space of $(1,0)$-vector fields of class $W^{l+1}$ on $X$ (for the structure $J_0$), whereas $T_{J_0}\mathcal E$ is the Hilbert space of $(0,1)$-forms of same class with values in $T$. The map $\phi$ is defined as the exponential of some analytic riemannian metric, see \cite{Kur1} for more details. Considering $G$ as a subset of the set of maps of class $W^{l+1}$ from $X$ onto the complex manifold $(X, J_0)$, then we get a $\mathbb C$-analytic structure on $G$ with chart \eqref{Gcarte}.
\vspace{5pt}\\
The space $E$ is the vector space of $J_0$-holomorphic vector fields. We also have
$$
L(\xi,\omega)=\omega+\bar\partial \xi
$$
so that $P=\bar\partial$. It is known to be an elliptic operator on vector fields and so (H2diff) and (H3diff) are satisfied. It can also be easily checked that hypotheses (H4) and (H5) are satisfied (just take the time 1 flow of a vector field as map $\psi$; since $G_{J_0}$ is the automorphism group of $(X,J_0)$, by definition, it contains only holomorphic, hence $C^\infty$, maps). As a consequence, we may apply Theorem \ref{main} and there is a minimal local section $K$. Using remark \ref{star} and the integrability condition given in \cite{Kur1}, it is given by
\begin{equation}
\label{Kcomplexe}
K=\{\omega\in U\mid  \bar\partial\omega+\dfrac{1}{2}[\omega,\omega]=\bar\partial^*\omega=0\}.
\end{equation}
It is important to notice that $K$ is not only a closed subset but also has a natural structure of analytic space with Zariski tangent space \eqref{Zariski} at the the base point. More precisely, it is given by the kernel of the $\bar\partial$-laplacian, a strongly elliptic operator. Hence it is finite dimensional and contains only $C^\infty$ solutions (and there is no dependance at all in the class $l$ which explains that we denoted it as $K$ and not $K^l$).
\vspace{5pt}\\
Moreover,
(FG) property is true by Fischer-Grauert's theorem, so $K$ is indeed a local moduli section.
\vspace{5pt}\\
Last, but not least, it follows directly from comparing \eqref{Kcomplexe} with \cite{Kur1} that the germ of $K$ at $J_0$ is the Kuranishi space of $(X,J_0)$ in the classical sense. Indeed, it is proven in \cite{modules} that versality is equivalent to being a local moduli section.
\vspace{5pt}\\
Finally, it is known (\cite{Wa} or \cite{KuranishiNote}), that $K$ is not in general a local moduli space, but that it 
is as soon as the dimension of the space of $J$-holomorphic vector fields on $X$ is constant when $J$ varies in $K$. 
\vspace{5pt}\\
Let us compare with Proposition \ref{LMS}.  Indeed, \cite{KuranishiNote} contains the construction of a map $\psi$ satisfying (H4'). And (H5') follows easily, taking into account that all isotropy groups contain only $C^\infty$ elements. So we have

\begin{proposition}
\label{complex}
 Consider the case of complex structures. Then,
 \begin{enumerate}
  \item Properties {\rm (H2diff), (H3diff), (H4), (H5)} as well as {\rm (FG)} are always satisfied.
  \item Any Kuranishi type space $K$ is isomorphic (as a germ) to the Kuranishi space of $(X, J_0)$.
  \item The same $K$ is a Kuranishi type space for both smooth and Sobolev structures.
  \item Conditions {\rm (H4')} and {\rm (H5')} are satisfied if and only if the dimension of the space of $J$-holomorphic vector fields on $X$ is constant when $J$ varies in $K$.
 \end{enumerate}

\end{proposition}

\begin{remark}
\label{Douadyreduced}
To be precise, $K$ as an analytic space is not always reduced, hence does not always identifies with $K$ as an analytic set. Hence, there are slight differences between Proposition \ref{complex} and the results in the literature on deformations of complex structures. For example, Douady proved that isomorphism \eqref{Kuranishi} is indeed an isomorphism of Banach $\mathbb C$-analytic spaces, cf. \cite{Douady}. Here, we just recover the isomorphism between the {\it reductions} of the involved spaces. Indeed, to avoid all the difficulties, one can read Proposition \ref{complex} replacing $K$ with its reduction. 
\medskip\\
However, it must be noticed that point (ii) of Proposition \ref{complex}, namely the equivalence between Kuranishi space and Kuranishi type space, is shown to be an isomorphism of analytic spaces even in the non-reduced case in \cite{modules} by imposing that \eqref{retraction} is analytic.
\end{remark}

\begin{remark}
Corollary \ref{H0=0} is nothing else in this context that the statement: if $H^0(X_J,\Theta_J)$ is zero for all $J\in K$, then the Kuranishi space is a local moduli space (also called universal). Indeed, due to the semi-continuity results of \cite{K-S}, it is enough to have $H^0(X_{J_0},\Theta_{J_0})$ equal to zero.
\end{remark}

\begin{remark}
If we consider the problem of deforming  couples (complex structure, additional geometric structure), assuming that $P$ is still a differential operator, then $P$ is automatically elliptic since the first component of its symbol is injective. Moreover, the automorphism group of the base structure contains only holomorphic transformations. Hence hypotheses (H2diff), (H3diff), (H4) and (H5) are automatically satisfied, and there always exists a Kuranishi type space. This applies for example to the case of symplectic holomorphic structures.
\end{remark}

\subsection{Riemannian metrics}
\label{riemann}
The case of riemannian metrics on a smooth compact manifold $X$ is due to Ebin \cite{Ebin}. It perfectly fits to this setting. Here $\mathcal E=\mathcal I$ is the set of $W^l$ riemannian metrics on $X$, encoded as the open positive convex cone of definite positive symmetric $2$-tensors. This is an open set of the Hilbert space of symmetric contravariant $2$-tensors. The group $G$ is the set of diffeomorphisms of class $W^{l+1}$ of $X$ acting by pullback on $\mathcal E$, so $T$ is just the vector space of $W^{l+1}$-vector fields on $X$. Let $g_0\in\mathcal E$. The (riemannian) exponential map associated to $g_0$ can be used as map $\phi$. By a direct computation,
$$
(\xi, h)\in T\times T_{g_0}\mathcal E\longmapsto L(\xi,h)=h+\mathcal L_\xi g_0
$$
where $\mathcal L$ is the Lie derivative (cf. \cite{Ebin}, Lemma 6.2). So $P$ is just the Lie derivative of $g_0$. It is elliptic by \cite{Ebin}, Proposition 6.10, hence (H2diff) and (H3diff) are satisfied. 
\vspace{5pt}\\
Hence, we may apply proposition \ref{lcexistence} and obtain a local section. Also (H4) is satisfied as well as (H5) by defining $\psi$ as the exponential map associated to $g_0$. Therefore the local section of \cite{Ebin} is a Kuranishi type space.
\vspace{5pt}\\
Moreover, the Kuranishi type space $K$ of \cite{Ebin} enjoys the following property: if $f$ is a diffeomorphism such that $K\cdot f$ intersects $K$, then $f$ must be an isometry of $g_0$, \cite{Ebin}, Theorem 7.1. This implies (FG) property, since if $c$ is a continuous path of $K$ all of whose points encode $g_0$, then all points are in fact equal to $g_0\cdot\phi$, with $\phi$ an isometry of $g_0$. Hence the path $c$ is constant.
\vspace{5pt}\\
Finally, it is proven in \cite{Ebin} that if the isotropy group of $g_0$ is the identity, it is still the identity for $g$ close to $g_0$. So in this case, (H4') and (H5') are satisfied and we obtain a local moduli space (this is indeed a direct application of Corollary \ref{H0=0}). Last but not least, it is proven in \cite{Ebin}, Theorem 7.4, that the result are still valid for $C^\infty$ metrics by taking as $K$ the subset 
of $C^\infty$ points of $K^l$. Indeed, Ebin shows the existence of a smooth invariant riemannian metric on $\mathcal E$, hence a smooth invariant affine connection. The result follows now from
Proposition \ref{Cinfini}. Hypothesis (H6) is proved in \cite{Ebin}, Proposition 6.13. To sum up,

\begin{proposition}
\label{Ebincase}
 Consider the case of riemannian metrics. Then,
 \begin{enumerate}
  \item Properties {\rm (H2diff), (H3diff), (H4), (H5)} are always satisfied, so given a riemannian metric $g_0$, for all $l$, it has a Kuranishi type space $K^l$ given as a neighborhood of $0$ in the kernel of $P^*$.
\item {\rm (FG)} property is fulfilled so $K^l$ is a local moduli section. 
  \item If the isotropy group of $g_0$ is the identity (which is the case on an open and dense subset of $\mathcal E$), then conditions {\rm (H4')} and  {\rm (H5')} are satisfied and $K^l$ is a local moduli space.
\item Properties {\rm (H6), (H7)} and {\rm (H8)} are satisfied. Hence, defining $K^\infty$ as the subset of $C^\infty$ points of $K^l$, then (1), (2) and (3) are still valid for $C^\infty$ metrics.
 \end{enumerate}

\end{proposition}

\begin{remark}
\label{remarkEbin}
In \cite{Ebin}, the author constructs two riemannian metrics on $\mathcal E^l$, the strong one and the weak one. Here, to run Proposition \ref{Cinfini}, we use  the strong riemannian metric, which depends on a particular choice of $l$. Ebin prefers using the weak one (weak in the sense that it induces on each tangent space to $\mathcal E^l$ the $L^2$ topology and not the required $W^l$ topology), because it has the advantage of being independent of $l$. However, with such a weak metric, the existence of the affine connection and the exponential map is not immediate.  
\end{remark} 

\subsection{ASD connections}
\label{ASD}
The case of ASD connections is due to Donaldson, see \cite[\S 4.2]{DonaldsonKronheimer}. Let $(X,h)$ be a compact, oriented, riemannian $4$-manifold. Let $E$ be a complex vector bundle over $X$ with first Chern class equal to zero. Define $\mathcal E$ as the space of connexions on $E$ compatible with $h$ and inducing the trivial connection on $\text{Det }E$. As usual, consider $W^l$ connections, for $l>1$. The space $\mathcal E$ is an affine space. Its associated vector space is the space of $1$-forms with values in $su(E)$. The set $\mathcal I$ of ASD connections is defined as those connections whose self-dual part of the curvature tensor is zero. Hence $\mathcal I$ is given as the zero set of a smooth map $F^+$ from $\mathcal A$ onto the vector space of $2$-forms with values in $su(2)$. The group $G$ is the gauge group of $E$ of class $W^{l+1}$, that is the group of sections of $SU(E,h)$. It acts on $\mathcal A$ by conjugation. We may take the Lie group exponential in the fibers as map $\phi$.\vspace{5pt}\\
At a point $A\in\mathcal A$, the operator $P$ is $-d_A$, for $d_A$ the covariant derivative going from the space of $W^{l+1}$ sections of $su(E)$ to the space of $W^l$ $1$-forms with values in $su(E)$. It is elliptic hence (H2diff) and (H3diff) are satisfied.\vspace{5pt}\\
The isotropy group of a connection $A$ is a finite-dimensional Lie group tangent to the kernel of the operator $d_A$. Hence (H4) is fulfilled and we may apply Proposition \ref{Pelliptic} to conclude that there exists a Kuranishi type space at $A$. 
Assume now that $A$ is an irreducible connection, that is with holonomy group being the full group $SU(2)$. Then its isotropy group is just $\pm Id$. Hence (H4') and (H5') are fulfilled and the Kuranishi type space is indeed a local moduli space. 
To sum up, we have
\begin{proposition}
 \label{ASDcase}
 Let $(X,h)$ be a compact, oriented, riemannian $4$-manifold. Let $E$ be a complex vector bundle over $X$ with first Chern class equal to zero. Consider $\mathcal E$, $\mathcal I$ an $G$ as above. Finally, let $A$ be an ASD connection. 
 
 Then, 
\begin{enumerate}
 \item Hypotheses {\rm (H2diff), (H3diff)} and {\rm (H4)} are fulfilled, hence the set
  \begin{equation}
  \label{KASD}
 K=\{\omega \in A^1(su(E))\mid F^+\omega=d_A^*\omega=0\}.
 \end{equation}
is a Kuranishi type space at $A$.
\item Assume moreover that $A$ is irreducible. Then {\rm (H4')} and {\rm (H5')} are fulfilled and {\rm \eqref{KASD}} is a local moduli space.
\end{enumerate}
 \end{proposition}

\begin{remark}
 \label{ASDsuite}
 Assume that $X$ is simply connected and $c_2(E)$ is not zero. Using Fredholm theory and a Sard type theorem, one then shows that $K$ is a finite-dimensional manifold if the metric $h$ is generic. Moreover, still by genericity of $h$, we may assume that all connections in $\mathcal I$ are irreducible. Hence the whole space $\mathcal I/G$ is a finite-dimensional manifold. Also, a direct argument shows that $K$ does not depend on $l$ up to homeomorphism, see \cite[\S 4.2]{DonaldsonKronheimer} for further details.
\end{remark}

\section{Deforming sasakian manifolds}
\label{sasakian}

\subsection{Background}
\label{background}

We start with some classical facts about sasakian manifolds, see \cite{BoyerBook} and \cite{Sparks} for more details.

\begin{definition}
\label{defsasakien}
 A compact smooth riemannian manifold $(S,g)$ is called a {\it sasakian manifold} if the cone $S\times\mathbb{R}^{>0}$ admits a complex structure which is 
 K\"ahler for the metric $r^2g+dr\otimes dr$ (where $r$ is the coordinate on $\mathbb{R}^{>0}$).
\end{definition}

A sasakian manifold comes equipped with many structures. Identifying $S$ with the hypersurface $S\times\{1\}$ of its cone and denoting by 
$J$ the complex operator on the cone, we have:

\begin{itemize}
 \item The unit vector field 
 $$
 \xi:=J\left (r\dfrac{\partial}{\partial r}\right )
 $$ 
 is tangent to $S$, acts by isometries on $(S,g)$, and is called the {\it Reeb vector field}.
 \item The contact form 
 $$
 \eta:=J\left (\dfrac{dr}{r}\right )
 $$ 
 is tangent to $S$ and satisfies 
 \begin{equation}
  \label{Reebcontact}
   i_\xi \eta\equiv 1\qquad\text{ and }\qquad i_\xi d\eta\equiv 0.
 \end{equation}
 \item The operator defined by 
 \begin{equation}
\label{Phi}
 \Phi(\xi):=0\qquad\text{ and }\qquad\Phi(V):=JV\text{ on Ker }\eta
 \end{equation}
 is an endomorphism of $TS$ which induces an integrable CR operator on $D:=\text{Ker }\eta$.
\end{itemize}

A sasakian manifold enjoys the following properties.

\begin{itemize}
 \item $\xi$ acts by CR isomorphisms, i.e. its flow preserves $D$ and $J$.
 \item The foliation $\mathcal F$ induced by $\xi$ on $S$ is transversely K\"ahler, with holomorphic normal bundle identified with $D$, and transverse K\"ahler 
 form $\omega:=d\eta$.
 \item The CR structure $(D, J)$ is strictly pseudo-convex with Levi form equal to $\omega$.
\end{itemize}

We denote by $(S,g,\xi,\eta, \Phi)$ a sasakian manifold. Note that 
\begin{equation}
 \label{JPhi}
 J\equiv \Phi+\dfrac{1}{r}dr\otimes \xi-\eta\otimes r\dfrac{\partial}{\partial r}.
\end{equation}
and, for $V$ tangent to $S$,
$$
\Phi(V)=J(V-\eta(V)\xi).
$$

All these datas are not independent, for example fixing $\eta$ and $g$ gives a unique $\xi$ through (\ref{Reebcontact}), and a unique $\Phi$
through
\begin{equation}
\label{nabla}
 \Phi(V)=\nabla_\xi V
\end{equation}
where $\nabla$ is the Levi-Civita connection of $g$. Moreover, $g$ and $\eta$ are related one to the other through the equations
\begin{equation}
 \label{geta}
 \eta(V)=g(\xi,V)\qquad\text{ and }\qquad g(V,W)=\dfrac{1}{2}d\eta(V,\Phi(W))+\eta(V)\eta(W).
\end{equation}

However for our deformation purposes, it is important to keep
track of these four structures, as well as the associated structures $\mathcal F$ and $(D,J)$ on $S$.
Indeed, when deforming sasakian structures, one has to be very careful and precise about which structure(s) is (are) fixed, and which is (are) deformed; 
and one has to
decide if we only consider deformations which are still sasakian or allow general deformations (in some problems there is no difference but not in all).
This is not always the case in the existing literature.
\vspace{5pt}\\
We first focus on the sasakian deformations of the contact form $\eta$, keeping the transverse holomorphic structure of $\mathcal F$ fixed. Recall that the normal bundle to $\mathcal F$ is the quotient bundle
$$
N\mathcal F:=TS/T\mathcal F
$$
and that the transverse holomorphic structure of $\mathcal F$ is determined by a splitting of the complexified normal bundle to the foliation
$$
N_{\mathbb C}\mathcal F:=N\mathcal F\otimes\mathbb C=N^{0,1}\mathcal F\oplus N^{1,0}\mathcal F
$$
into $(1,0)$ and $(0,1)$ vectors. The subbundles $N^{0,1}\mathcal F$ and $N^{1,0}\mathcal F$
are complex conjugate and involutive (for the quotient Lie bracket). Fixing the transverse holomorphic structure means keeping $N^{0,1}\mathcal F$ fixed.
\vspace{5pt}\\
Equivalently, looking at the natural projection map
\begin{equation}
\label{pi}
\pi\ :\ T_{\mathbb C}S\longrightarrow N_{\mathbb C}\mathcal F\longrightarrow N^{1,0}\mathcal F
\end{equation}
and setting
\begin{equation}
\label{E}
E=\text{Ker }\pi
\end{equation}
the transverse holomorphic structure is given by the involutive subbundle $E$ of $T_{\mathbb C}S$, so fixing the transverse holomorphic structure means fixing $E$ (cf. \cite{MarcelPS}). In the sasakian case, observe that
\begin{equation}
\label{transverse}
E=D^{0,1}\oplus \mathbb C\xi
\end{equation}
where $D^{0,1}$ is the subbundle of $(0,1)$-vectors of the complexification of the CR distribution $(D,J)$. In particular, given $E$, we have
$$
D^{0,1}=E\cap D_{\mathbb C}
$$
so $J$ is uniquely determined on $D_{\mathbb C}$ and thus on $D$: it acts as multiplication by $-i$ on $D^{0,1}$ and as
multiplication by $+i$ on its complex conjugate.
\vspace{5pt}\\
The reason for dealing with this problem is that this is perhaps the simplest case where infinite-dimensionality occurs, so that the classical theory of deformations does not 
apply. 
\vspace{5pt}\\
Observe that $E$ being fixed, $\xi$ is only changed by a multiplicative factor. As a variant to this deformation problem, one can deform $\eta$ keeping $E$ and $\xi$ fixed. The resulting Kuranishi type space will be essentially the same (see Corollary \ref{KuranishiEbisfixed}).

\begin{remark}
\label{Phifixed}
If $\Phi$ is fixed, then so is $E$, simply because 
$$
E=\text{Im }(Id+i\Phi).
$$
However, the converse is false. Indeed, $\Phi$ determines also $D^{0,1}$ as the kernel of $\Phi+iId$ (and thus $D$), whereas $E$ does not. In other words, fixing $E$ means fixing $N^{0,1}\mathcal F$, whereas fixing $\Phi$ means fixing $D^{0,1}$, which is a precise realization of $N^{0,1}\mathcal F$ as a subbundle of $T_\mathbb C S$. Deformations of $S$ with $\Phi$ fixed correspond to deformations of the induced polarized CR structure as defined and studied
in \cite{MeIMRN}.
\end{remark}

Note that we cannot fix more structures, since we have

\begin{lemma}
Let $(S,g,\eta,\xi,\Phi)$ and $(S,g',\eta',\xi',\Phi')$ be two sasakian structures on the same
manifold $S$. 
\begin{enumerate}
\item If $g=g'$ and $\xi'=\xi$, then $\eta'=\eta$ and $\Phi'=\Phi$ so both structures are the same.
\item If $\xi'=\xi$ and $\Phi'=\Phi$, then $\eta'=\eta$ and $g'=g$ so both structures are the same.
\end{enumerate}
\end{lemma}

\begin{proof}
If $g=g'$ and $\xi'=\xi$, then $\eta'$ is equal to $\eta$ because of (\ref{geta}) and $\Phi'$ is equal to $\Phi$ because of (\ref{nabla}).
\vspace{5pt}\\
And if we let $\Phi$ fixed, as observed in remark \ref{Phifixed}, then $D$ is fixed. Since $\xi$ is also fixed, then $\eta$ is fixed, since it is zero on $D$, and $1$ on $\xi$. Finally $g$
is fixed because of \eqref{geta}.
\end{proof}

\subsection{Encoding the structures}
\label{prelim}

We first need the following characterization of sasakian manifolds.

\begin{proposition}
\label{charsasakien}
Let $(S,\eta,J)$ be a  triple: compact smooth manifold, contact form, integrable CR-structure on the kernel $D$ of $\eta$. Define $\xi$ using {\rm (\ref{Reebcontact})}. Assume that
\begin{enumerate}
\item $\mathcal L_{\xi}J\equiv 0$.
\item $d\eta(V,JV)>0$ for all non zero $V$ tangent to $D$.
\end{enumerate}
Then, defining $\Phi$ by {\rm (\ref{Phi})} and $g$ by {\rm (\ref{geta})}, the manifold $(S,g,\eta,\xi,\Phi)$ is sasakian.
\end{proposition}

\begin{proof}
We will first show that $g$ (defined through \eqref{geta}) is a riemannian metric. The integrability of the CR-structure implies that, for all vector fields $X$ 
and $Y$ tangent to $D$, we have
\begin{equation}
 \label{1}
 [X+iJX,Y+iJY]=Z+iJZ
\end{equation}
for some $Z$ tangent to $D$. It follows from \eqref{1} that
\begin{equation}
 \label{2}
 \left\{
 \begin{aligned}
 &[JX,Y]+[X,JY]\in\Gamma(D)\cr
 &[X,Y]-[JX,JY]\in\Gamma(D)
 \end{aligned}
 \right .
\end{equation}
Using the fact that $D$ is the kernel of $\eta$ and applying $\eta$ to \eqref{2}, we obtain
\begin{equation}
 \label{3}
 d\eta(X,Y)=d\eta(JX,JY)
\end{equation}
and
\begin{equation}
 \label{4}
 d\eta(JX,Y)+d\eta(X,JY)=0.
\end{equation}
Now, \eqref{3} means that $d\eta$ is a $(1,1)$-form, and \eqref{4} means that $g$ is symmetric. By (2), this is enough to prove that $g$ is a riemannian
metric.
\vspace{5pt}\\
We consider now the riemannian cone $(S\times\mathbb{R}^{>0},\bar g=r^2g+dr\otimes dr)$. We extend $J$ to the cone by setting
\begin{equation}
\label{defJ}
Jr\dfrac{\partial}{\partial r}=\xi\quad\text{ and }\quad J\xi=-r\dfrac{\partial}{\partial r}.
\end{equation}
Notice that with this definition, condition~(1) in the statement of the proposition is still fulfilled. It is straightforward to check that the metric $\bar g$
is $J$-invariant. 
We also set
\begin{equation}
\label{defomega}
\omega:=\bar g(J-,-).
\end{equation}
We will now show that $J$ defines a complex structure on the cone. Indeed, the bundle of $(1,0)$-vectors, say $Q^{1,0}$, satisfies
$$
Q^{1,0}=D^{1,0}\oplus \mathbb C (\xi+ir\dfrac{\partial}{\partial r}).
$$
But, for $X$ tangent to $D$, we have
$$
[X-iJX,\xi+ir\dfrac{\partial}{\partial r}]=-\mathcal L_\xi X+i\mathcal L_\xi(JX)=-\mathcal L_\xi X+iJ\mathcal L_\xi X
$$
because of condition (1). So it belongs to $D^{1,0}$. Since this bundle is involutive, this proves the involutivity of $Q^{1,0}$.
\vspace{5pt}\\
Our last step is to prove that $d\omega$ is zero. Since $\omega$ is the $(1,1)$-form associated to the $J$-invariant riemannian metric $\bar g$, this shows
that $\bar g$ is a k\"ahlerian metric, so, by definition \ref{defsasakien}, we are done.
\vspace{5pt}\\
We first claim that 
\begin{equation}
\label{5}
i_\xi d\omega=0.
\end{equation}
This can be proven as follows. Take $X$ and $Y$ local vector fields tangent to $D$ and commuting with $\xi$. Using the relations
$$
\mathcal L_\xi \eta=\mathcal L_\xi d\eta=\mathcal L_\xi X=\mathcal L_\xi Y=\mathcal L_\xi J=0
$$
we deduce that 
\begin{equation}
\label{6}
(\mathcal L_\xi g)(X,Y)=\mathcal L_\xi (g(X,Y))=\mathcal L_\xi(1/2d\eta(X,JY))=0.
\end{equation}
Similar computations replacing $(X,Y)$ with $(X,\xi)$ and then $(\xi,\xi)$ show that
\begin{equation}
\label{7}
\mathcal L_\xi g=0 \quad \text{and thus}\quad \mathcal L_\xi \bar g=r^2\mathcal L_\xi g+\mathcal L_\xi(dr\otimes dr)=0
\end{equation}
and from \eqref{7} that
\begin{equation}
\label{8}
\mathcal L_\xi\omega=0.
\end{equation}
Moreover,
$$
i_\xi\omega(X)=\bar g(-r\dfrac{\partial}{\partial r},X)=0
$$
and
$$
i_\xi\omega(r\dfrac{\partial}{\partial r})=g(-r\dfrac{\partial}{\partial r},r\dfrac{\partial}{\partial r})=- r^2
$$
yielding
\begin{equation}
\label{9}
i_\xi\omega=-rdr.
\end{equation}
Combining \eqref{8} and \eqref{9} gives \eqref{5}.
\vspace{5pt}\\
We are now in position to show that $d\omega$ is zero and thus to finish with the proof. Because of identity \eqref{5},
it is sufficient to see that $d\omega$ vanishes on $D \oplus\langle r\frac{\partial}{\partial r}\rangle$.  
Choose local coordinates $(t, x_i)$ on $S$ such that $\xi = \frac{\partial}{\partial t}$, and let $Y_i= \frac{\partial}{\partial x_i} + a_i\,\xi$ 
be the local vector fields defined by the condition $\eta(Y_i)=0$. Then $\{Y_i\}$ is a local basis of $D$ and the following 
identities are fulfilled
%
%
%
%
\begin{equation}
\label{10}
[Y_i,r\dfrac{\partial}{\partial r}]=[Y_i,\xi]=0,
\end{equation}
and
\begin{equation}
\label{10bis}
[Y_i,Y_j]= \eta([Y_i, Y_j])\, \xi = - d\eta(Y_i, Y_j)\, \xi.
\end{equation}
This last identity gives 
\begin{equation}
\label{10tris}
\bar g([Y_i,Y_j], Y_k) =0  \quad \text{and} \quad \omega([Y_i,Y_j], Y_k) =0,
\end{equation}
which imply the following relation,
$$
d\omega(Y_i,Y_j,Y_k)=Y_i\cdot \omega (Y_j,Y_k)-Y_j\cdot\omega (Y_i,Y_k)+Y_k\cdot
\omega (Y_i,Y_j).
$$
From that identity and using the definition of $\bar g$ and $\omega$, 
as well as formulas \eqref{3}, 
\eqref{10bis} and \eqref{10tris},
we obtain 
$$
\begin{array}{rcl}
d\omega(Y_i,Y_j,Y_k) & = & Y_i\, \bar g (JY_j, Y_k) -  Y_j\,\bar g  (JY_i, Y_k) + Y_k\, \bar g  (JY_i, Y_j) \\[3mm]
%
 & = & \dfrac{r^2}{2} \big[  Y_i\, d\eta (Y_j, Y_k) -  Y_j\, d\eta (Y_i, Y_k) + Y_k\, d\eta (Y_i, Y_j) \big] \\[3mm]
 & = & \dfrac{r^2}{2} d(d\eta) (Y_i,Y_j,Y_k) =0.
\end{array}
$$
Finally, since $D$, $\langle \xi \rangle$ and $\langle r\dfrac{\partial}{\partial r} \rangle$ are mutually $\bar g$-orthogonal
and using \eqref{10}, \eqref{10bis} and \eqref{3}, we deduce
$$
\begin{array}{rcl}
d\omega(r\dfrac{\partial}{\partial r}, Y_i, Y_j)  & = & 
r\dfrac{\partial}{\partial r}\bar g(J Y_i, Y_j) - Y_i \bar g (J r\dfrac{\partial}{\partial r}, Y_j ) + Y_j \bar g (J r\dfrac{\partial}{\partial r}, Y_i) \\
 &  & \quad - \omega([r\dfrac{\partial}{\partial r}, Y_i], Y_j) +  \omega([r\dfrac{\partial}{\partial r}, Y_j], Y_i) - \omega([Y_i, Y_j], r\dfrac{\partial}{\partial r})\\[2mm]
 & = & r\dfrac{\partial}{\partial r} (r^2 g(JY_i, Y_j)) - \bar g(J[Y_i, Y_j], r\dfrac{\partial}{\partial r})) \\[2mm]
 & = & 2 r^2 \big(\dfrac12 d\eta (JY_i, JY_j)) - \bar g(- J(d\eta(Y_i, Y_j)\, \xi), r\dfrac{\partial}{\partial r}) \\[2mm]
 & = &  r^2 d\eta (Y_i, Y_j) - \bar g (d\eta(Y_i, Y_j) r\dfrac{\partial}{\partial r} , r\dfrac{\partial}{\partial r} ) = 0, 
\end{array}
$$
%
%
so $d\omega$ is zero and we are done.
\end{proof}

\begin{corollary}
\label{Echar}
Let $S$ be a compact smooth manifold. Then, a sasakian structure on $S$ determines uniquely and is uniquely determined by the data of a subbundle $E$ of the complexified tangent bundle and a contact form $\eta$ satisfying
\begin{enumerate}
\item $E+\bar E=T_{\mathbb C}S$.
\item $E\cap \bar E=\mathbb C\xi$.
\item $[E,E]\subset E$.
\item The form $d\eta$ satisfies $\restriction{d\eta}{E}=0$.
\item For all non-zero vector $V$ of $E\cap D_{\mathbb C}$, one has
$d\eta(V, i\bar V)+d\eta(\bar V,-iV)>0$.
\end{enumerate}
where $\xi$ is defined through {\rm (\ref{Reebcontact})} and $D$ is the kernel of $\eta$. 
\end{corollary}
\begin{remark}
In the proof of corollary \ref{Echar}, the sasakian structure will be made explicit from $E$ and $\eta$. 
\end{remark}

\begin{proof}
Let $(g,\eta,\xi,\Phi)$ be a sasakian structure on $S$. Define $E$ through (\ref{transverse}). Then the conditions above are easily verified.
\vspace{5pt}\\
Conversely, let $(E,\eta)$ be as above. By a result of \cite{Nirenberg}, the first three conditions imply that the foliation $\mathcal F$ induced by $\xi$ is transversely holomorphic. Set 
$$
D^{0,1}:=D_{\mathbb C}\cap E.
$$
We have
$$
T_{\mathbb C}S=D_{\mathbb C}\oplus \mathbb C\xi\Longrightarrow E=D^{0,1}\oplus\mathbb C\xi.
$$
The fourth and fifth conditions can now be rewritten as: $d\eta$ is a basic $(1,1)$-form and $d\eta(V, JV)>0$ for all non zero V tangent to $D$. Moreover, for $V$ and $W$ tangent to $D^{0,1}$,
$$
\eta([V,W])=d\eta (V,W)=0
$$
since $d\eta$ is $(1,1)$. So $[V,W]$ belongs to $D_{\mathbb C}$, hence to $D^{0,1}$, which is thus involutive, proving the integrability of $(D,J)$.
\vspace{5pt}\\
Notice that $V$ tangent to $D$ implies that $[\xi, D]$ is tangent to $D$, since
$$
\eta([\xi, V])=d\eta(\xi, V)=0.
$$
Finally, the transverse holomorphic structure is by definition invariant by holonomy, hence we have 
\begin{equation}
\label{Jinvariance}
[\xi, JV]-J[\xi, V]\in\mathbb C\xi.
\end{equation}
Since \eqref{Jinvariance} is tangent to $D$ by the previous remark, it must be zero. This can be rephrased as: $\mathcal L_\xi J$ is zero.
Applying proposition \ref{charsasakien} yields the result.
\end{proof}

We are now in position to give a good encoding of the sasakian deformations with the transverse holomorphic structure fixed.

\begin{corollary}
\label{alpha}
Let $(S,g,\xi,\eta, \Phi)$ be a sasakian manifold. Define $E$ as in {\rm (\ref{E})}. Then, there exists a neighborhood $U_0$ of $0$ in the space of  real $1$-forms on $S$ such that, for all $\alpha\in U_0$, the following conditions are equivalent
\begin{enumerate}
\item The triple $(S,E, \eta+\alpha)$ is a sasakian manifold.
\item We have $\restriction{d\alpha}{E}=0$.
\end{enumerate}
\end{corollary}

\begin{remark}
Of course, to make Corollary \ref{alpha} precise, one has to fix the same regularity on the space of $1$-forms and basic $1$-forms: $C^\infty$ or $W^l$, etc...
\end{remark}

\begin{proof}
Choose $U_0$ so that, for all $\alpha\in U_0$, we have that $\eta+\alpha$ is a contact form and the fifth condition of corollary \ref{Echar} is fulfilled. This is possible since both properties are open.
\vspace{5pt}\\
Assume that $(S,E,\eta')$ is sasakian, with $\eta'-\eta$ in $U_0$. Define $\alpha:=\eta'-\eta$. Let $\xi'$ be the Reeb vector field associated to $\eta'$ through \eqref{Reebcontact}. Since $E$ is kept fixed, it follows from Corollary \ref{Echar} that $\restriction{d\alpha}{E}$ is zero.
\vspace{5pt}\\
Conversely, let $\alpha$ belong to $U_0$ and satisfy {\it (2)}. Then, for 
$$
V=\lambda\xi\oplus W\oplus\bar W
$$ 
a real vector field (here with $W$ tangent to $D^{1,0}$), we have 
$$
\label{p}
i_\xi d(\eta+\alpha)(V)=i_\xi d\alpha(V)=d\alpha(\xi,W)+d\alpha(\xi, \bar W)=0
$$
because of {\it (2)} and because $\bar E\cap E=\mathbb C\xi$. Hence, the Reeb vector field $\xi'$ associated to $\eta'$
is a multiple of $\xi$. Applying corollary \ref{Echar} gives the result.
\end{proof}

We are now in position to prove a existence of the Kuranishi type space for the deformations of $\eta$ with the transverse structure of $\mathcal F$ fixed.

\subsection{Deformations of the contact form of a sasakian structure}
\label{Kuretasasakien}

In this subsection, we construct a Kuranishi type space for the deformations of the contact form of a sasakian manifold.
Let $(S,E,\eta)$ be a sasakian manifold. Fix $l>1+\dim S/2$ and consider only $1$-forms of class $W^l$.  We let $T^*S$ be the cotangent bundle
of $S$. The notation $\Gamma^l(-)$ stands for the vector space of $W^l$ sections of the corresponding bundle. Let 
\begin{equation}
\label{Icontact}
\mathcal E=\mathcal I=\{\alpha\in \Gamma^l(T^*S)\mid\restriction{d\alpha}{E}=0\}.
\end{equation}

Let $G=G^l$ be the connected component of the identity of the topological group of diffeomorphisms $f$ of $S$ of class $W^{l+1}$ such that
$$
f^*E\equiv E
$$
and let $G^\infty$ be the group of elements of class $C^\infty$.
\begin{lemma}
\label{Banach}
There exists a Fr\'echet chart {\rm \eqref{Gcarte}} of $G^\infty$ at $e$ which extends as a Hilbert chart of $G^l$.
\end{lemma}

\begin{proof}
First, consider the subgroup $G^\infty_0$ of elements of $G^\infty$ that preserve each leaf. Using a riemannian metric on $S$ (for example its sasakian metric), we obtain a riemannian exponential map, say $\phi_0$, modelling $G^\infty_0$ at identity on the vector space of smooth vector fields tangent to the foliation (here it is just the space $\mathcal X_0$ of multiples of $\xi$). Let $\mathcal X$ be the set of vector fields generating isomorphisms of the transversal holomorphic foliation $\mathcal F$. Then it decomposes naturally as
\begin{equation}\label{vectTH}
\mathcal X=\mathcal X_0\oplus \mathcal X_N
\end{equation}
where 
$\mathcal X_N$ is the space of holomorphic basic vector fields orthogonal to $\xi$.  Letting $\exp$ denote the exponential of Lie groups, we see that we can take the chart
$$
\chi=h\xi+\chi_N\in\mathcal X_0\times \mathcal X_N\longmapsto \exp(\chi_N)\circ\phi_0(h\xi)\in G^\infty
$$
as a Fr\'echet chart \eqref{Gcarte}. Passing to the Sobolev completions, it extends as a Hilbert chart (\ref{Gcarte}) of $G^l$.
\end{proof}

We use the notations of section \ref{GS}. In particular, $V$ (respectively $U$) denotes a neighborhood of $0$ in $T=\mathcal X$ (respectively $\mathcal E$). Action (\ref{action}) is given by
$$
(\chi,\alpha)\in V\times U\longmapsto \alpha\cdot\phi(\chi)=\phi(\chi)^*(\eta+\alpha)-\eta\in\mathcal E.
$$
Looking at the differential at $(0,\eta)$, 

\begin{lemma}
\label{Llemma}
We have 
$$
L(\chi,\alpha)=\alpha+P\chi=\alpha+\mathcal L_\chi \eta.
$$
\end{lemma}

\begin{proof}
By definition, we have
\begin{equation}
\label{calculL}
L(\chi,\alpha)=\restriction{\dfrac{d}{ds}}{s=0}\left (\phi(s\chi)^*(\eta+s\alpha)-\eta\right).
\end{equation}
Observe now that, in a local chart, for $s$ sufficiently small, we have
\begin{equation}
\label{DL}
\phi(s\chi)=Id+s\chi+\epsilon(s)
\end{equation}
Now, \eqref{DL} is exactly the Taylor development in $s$ of the flow of $\chi$ at $s$. Hence, {\it up to order $1$}, $\phi(s\chi)$ coincide with the flow $\phi_s^\chi$ of $\chi$ at time $s$. As a consequence, we immediately deduce from \eqref{calculL} that
$$
L(\chi,\alpha)=\alpha+ \restriction{\dfrac{d}{ds}}{s=0}\left ((\phi_s^\chi)^*\eta\right )
$$
and thus
$$
L(\chi,\alpha)=\alpha+P\chi=\alpha+\mathcal L_\chi \eta.
$$
\end{proof}
Finally,

\begin{lemma}
The image of $P$ is closed in $\mathcal E$.
\end{lemma}

\begin{proof}
Let $\chi\in V$. Set
$$
\chi=h\xi+\chi_N.
$$
We have
$$
\mathcal L_\chi \eta=d(i_{h\xi}\eta)+hi_\xi d\eta+d(i_{\chi_N}\eta)+i_{\chi_N}d\eta.
$$
Now, the third term is zero because the kernel of $\eta$ is generated by the vector fields $\chi_N$. Using \eqref{Reebcontact}, 
it follows that
$$
\mathcal L_\chi \eta=dh+i_{\chi_N}d\eta.
$$
But this formula shows that the image of $P$ is the sum of the image of the de Rham differential applied to the set of functions and of the image of a finite-dimensional vector space under a bounded linear operator. Hence it is the sum of a closed subspace and of a finite-dimensional one. So it is closed.
\end{proof}
Hence, hypotheses (H1), (H2), (H3), (H4) and (H5) are satisfied. Observe that the isotropy group of $\eta$ is the automorphism group of the sasakian manifold $(S,E,\eta)$. Hence it is finite-dimensional and we can use the time $1$ flow as chart $\psi$ fulfilling (H4). Define
\begin{equation}
\label{Kspace}
K_\eta^l:=\{\alpha\in U_0\mid P^*\alpha=\restriction{d\alpha}{E}=0\}
\end{equation}
we now have, using Theorem \ref{main},

\begin{theorem}
The space $K_\eta^l$ defined in {\rm \eqref{Kspace}} is an open neighborhood of $0$ in a infinite-dimensional Hilbert space and is a Kuranishi type space for $\eta$.
\end{theorem}

\begin{proof}
The equations in \eqref{Kspace} are all linear and continuous, hence $K_\eta^l$ 
is an open neighborhood of $0$ in a Hilbert space. Besides, it contains all the basic $1$-forms whose 
differential is $(1,1)$, and in particular all the $\partial\bar\partial f$ for $f$ a basic function. Finally, the Reeb flow of a sasakian manifold has no dense orbit, cf. \cite{BoyerBook} or \cite{Sparks}. Hence \eqref{Kspace} is infinite-dimensional.
\end{proof}

\begin{lemma}
\label{H89}
Hypotheses {\rm (H6), (H7)} and {\rm (H8)} are satisfied.
\end{lemma}

\begin{proof}
Assume that $f^*\eta-\eta$ and $\eta$ are $C^\infty$, with $f$ of class $W^l$ preserving $E$. We want to prove that $f$ is indeed $C^\infty$.
\vspace{5pt}\\
From $f^*\eta$ of class $C^\infty$ and \eqref{Reebcontact}, we deduce that $f_*\xi$ is also $C^\infty$. Moreover, since
\begin{equation}
\label{fD01}
f^*D^{0,1}=f^* E\cap \text{Ker }(f^*\eta) =E\cap \text{Ker }(f^*\eta)
\end{equation}
we have that $f\cdot \Phi$ is $C^\infty$. Finally \eqref{geta} implies that $f^*g$ is $C^\infty$. In particular, $f$ sends geodesics onto geodesics, a property that classically implies that $f$ is $C^\infty$.
\vspace{5pt}\\
Let us focus now on (H7). We construct an invariant riemannian metric on $\mathcal E$.
Each structure $\alpha\in\mathcal E$ encodes a unique sasakian metric $g_\alpha$ on $S$ through \eqref{geta}. This induces a unique riemannian metric
on the cotangent bundle of $S$, still denoted by $g_\alpha$. By integrating over $S$, we obtain a scalar product $h_\alpha$ on the space of
$1$-forms on $S$. The collection $(h_\alpha)$ is a riemannian metric on $\mathcal E$. It is obviously invariant under the action of the diffeomorphism group $G^\infty$. To show it is smooth,
we proceed as follows. Given $\alpha\in\mathcal E$, we define $D_\alpha$ as the kernel of $\eta+\alpha$, then $D^{0,1}_\alpha$ as the intersection of the complexification of $D_\alpha$ with $E$. This allows us to define $\Phi_\alpha$ and finally $g_\alpha$ through \eqref{geta}.

In this process, observe that
\begin{description}
 \item[(i)] To know $g_\alpha$ in a point $x\in M$, it is enough to know $\alpha(x)$ and $d\alpha(x)$.
 \item[(ii)] If $\alpha$ varies smoothly, then so does $g_\alpha$.
\end{description}

In other words, the map $\alpha\mapsto g_\alpha$ is a map from the $W^l$ sections of the bundle of $1$-jets of differential forms of degree $1$ into the $W^l$ sections of the bundle of symmetric $2$-tensors, which comes from a smooth vector bundle map. This is enough to show that it is smooth. It is then easy to see that $\alpha\mapsto h_\alpha$ is smooth, cf. \cite[p.18-19]{Ebin}. This gives a weak invariant metric. To get a strong one, one simply has to play the same game to obtain a weak metric on the bundle of $l$-jets of sections of $\mathcal E$, see \cite[p.20]{Ebin}. To be more precise, the weak invariant metric induces a weak invariant metric on each associated tensor bundle, and thus induces a weak riemannian metric on the bundle of $l$-jets of sections of $\mathcal E$. But this is equivalent to endowing the bundle of $W^l$-sections of $\mathcal E$ with a strong riemannian metric.
\vspace{5pt}\\
Finally Hypothesis (H8) is immediate. \end{proof}

 As an application of Lemma \ref{H89} and Theorem \ref{mainsmooth}, we thus have

\begin{corollary}
\label{KuranishiEfixed}
 Let $K_\eta^\infty$ be the subset of $C^\infty$ points of $K_\eta^l$. Then it is a Kuranishi type space for $C^\infty$ structures.
\end{corollary}

We can give a more precise description of \eqref{Kspace} by computing $P^*$. We can rewrite $P$ as
\begin{equation}
\label{Pbis}
(h,\chi)\in \Gamma^{l+1}(\mathbb R)\times{\mathcal X}_N\longmapsto dh+i_{\chi_N}d\eta\in\mathcal E.
\end{equation}
As usual, we let $g$ denote the sasakian metric of the base structure. In what follows, we extend $g$ to the $1$-forms and all the tensor fields. Then we use the $L^2$ product associated to $g$ on the tensor fields. In particular, on $\Gamma^{l+1}(\mathbb R)\times{\mathcal X}_N$, we use the sum of the $L^2$ product on the functions and that on the vector fields. We denote this sum as well as all the $L^2$ products by the same symbol $\langle -,-\rangle$.
\vspace{5pt}\\
Going back to \eqref{Pbis}, observe that the sum in the right expression is not a direct sum. But, defining the closed vector subspace
$$
\mathcal X'_N:=\left (\left (\restriction{P}{\{0\}\times\mathcal X_N}\right )^{-1}(P(\Gamma^{l+1}(\mathbb R)\times\{0\}))\right )^\perp
$$
then \eqref{Pbis} becomes
\begin{equation}
\label{Pter}
(h,\chi)\in \Gamma^l(\mathbb R)\times{\mathcal X'}_N\longmapsto dh\oplus i_{\chi_N}d\eta\in\mathcal E.
\end{equation}
Define now
\begin{equation}
\label{diese}
\beta\in \Gamma^l (T^*S)\longmapsto \beta^\sharp\in\Gamma^l(TS)
\end{equation}
by
\begin{equation}
\label{diesedef}
g(\beta^\sharp,-)=\beta.
\end{equation}
Observe that, with this convention,
\begin{equation}
\label{sharpproperty}
g(\alpha,\beta)=g(\alpha^\sharp,\beta^\sharp)=\alpha(\beta^\sharp)=\beta(\alpha^\sharp).
\end{equation}
We have
\begin{lemma}
\label{adjointeta}
 The adjoint of \eqref{Pter} is given by the formula
 $$
 \label{Padjointeta}
 P^*\alpha=(d^*\alpha,-(i_{\alpha^\sharp}d\eta)^\sharp)
 $$
 for $\alpha\in\mathcal E$ and $d^*$ the codifferential on $1$-forms.
\end{lemma}

\begin{proof}
 Just compute 
 $$
 \langle P^*\alpha,(h,\chi)\rangle=\langle d^*\alpha,h\rangle-\langle (i_{\alpha^\sharp}d\eta)^\sharp),\chi\rangle=\langle \alpha,dh\rangle
 - \int_Sd\eta (\alpha^\sharp, \chi)vol_g
 $$
 because of \eqref{diesedef} and of \eqref{sharpproperty}. But this is exactly
 $$
 \langle \alpha,dh\rangle+\int_S i_{\chi_N} d\eta (\alpha^\sharp )vol_g=\langle \alpha,dh\rangle
 +\langle \alpha, i_{\chi_N} d\eta\rangle
 $$
 finishing the proof.
\end{proof}

Let us treat rapidly the associated case where the contact form is deformed, keeping $E$ and $\xi$ fixed. The following statement is analogous to Corollary \ref{alpha} and is easy to prove.

\begin{corollary}
Let $(S,g,\xi,\eta, \Phi)$ be a sasakian manifold. Define $E$ as in {\rm (\ref{E})}. Then, there exists a neighborhood $U_0$ of $0$ in the space of $1$-forms on $S$ such that, for all $\alpha\in U_0$, the following conditions are equivalent
\begin{enumerate}
\item The triple $(S,E, \eta+\alpha)$ is a sasakian manifold with Reeb vector field $\xi$.
\item The $1$-form $\alpha$ is basic and its differential is $(1,1)$, that is satisfies
$$
i_\xi\alpha=\restriction{d\alpha}{E}=0.
$$
\end{enumerate}
\end{corollary}

Let $G$ be the topological group of diffeomorphisms $f$ of $S$ of class $W^{l+1}$ such that
$$
f_*\xi\equiv \xi\qquad\text{ and }\qquad f_*E\equiv E.
$$
With this new statement and this new group on mind, one obtains easily the

\begin{corollary}
The space $(K')^l_\eta$ defined as 
$$
(K')^l_\eta:=\{\alpha\in U_0\mid P^*\alpha=i_\xi\alpha=\restriction{d\alpha}{E}=0\}
$$
is an open neighborhood of $0$ in a infinite-dimensional Hilbert space and is a Kuranishi type space for $\eta$.
\end{corollary}

Here, the operator $P$ is the same as that appearing in Lemma \ref{Llemma}, but restricted to the subspace $\mathcal X_b$ of vector fields of $\mathcal X$ whose $\xi$-coordinate is basic. This implies that $P^*$ is slightly different from that of \eqref{Kspace}. Indeed it is the composition of this latter operator with the projection onto $\mathcal X_b$.

Also, using the same arguments as above,
\begin{corollary}
\label{KuranishiEbisfixed}
 Let $(K')_\eta^\infty$ be the subset of $C^\infty$ points of $(K')_\eta^l$. Then it is a Kuranishi type space for $C^\infty$ structures.
\end{corollary}

Finally, we have 

\begin{proposition}
\label{sasakienH0=0}
Assume that the Lie algebra $\mathcal X_N$ is zero. Then $K^l_\eta$ and $K^\infty_\eta$ (respectively $(K')_\eta^l$ and $(K')^\infty_\eta$) are local moduli spaces.
\end{proposition}

\begin{remark}
By \cite{DK}, this is equivalent to saying that the group of holomorphic basic infinitesimal automorphisms $H^0(S,\Theta)$ is zero.
\end{remark}

\begin{proof}
The automorphism group of $(E,\eta)$ consists of diffeomorphisms fixing $E$ and $\eta$. It is a finite dimensional Lie group whose Lie algebra consists of vector fields
$$
\chi=h\xi\oplus \chi_N\quad\text{ such that }\quad [\chi, E]\subset E,\ \mathcal L_\chi \eta=0.
$$
The $E$-preservation implies that $\chi_N$ belongs to $\mathcal X_N$.
\vspace{5pt}\\
If $\mathcal X_N$ is zero, then $\chi$ is a multiple of $\xi$ and, since
$$
\mathcal L_\chi\eta=dh,
$$
it is a constant multiple. Hence this Lie algebra is reduced to constant multiples of $\xi$ and is one-dimensional. 
\vspace{5pt}\\
Consider firstly the case of $E$ and $\xi$ fixed. For any other structure $(E,\eta')$, the Lie algebra of infinitesimal automorphisms is still equal to $\mathbb C\cdot\xi$. Hence, defining $\psi(\eta',\lambda\xi)$ as the time $1$ flow of the vector field $\lambda\xi$, we immediately have (H4') and (H5') fulfilled. The conclusion follows from Theorem \ref{main} and Theorem \ref{mainsmooth}.
\vspace{5pt}\\
Consider secondly the case of $E$ fixed. For any other structure $(E,\eta')$, since we keep the same $E$, the Lie algebra of infinitesimal automorphisms it is still one-dimensional and generated by the constant multiples of the corresponding Reeb vector field $\xi'$. Observe that the mapping $\eta'\mapsto \xi'$ is smooth. Hence, defining $\psi(\eta',\lambda\xi)$ as the time $1$ flow of the vector field $\lambda\xi'$, it is easy to check that (H4'), and (H5') are satisfied. So, once again, we may apply Theorem \ref{main} and Theorem \ref{mainsmooth}.
\end{proof} 

\subsection{General deformations}
\label{general}
We now deal with the case of general deformations of sasakian manifolds. Using corollary \ref{Echar}, this means deforming both $E$ and $\eta$.
\vspace{5pt}\\
Let $S$ be a compact smooth manifold of dimension $2n+1$. Let $\mathcal G$ be the grassmannian bundle of complex $(n+1)$-planes of $T_{\mathbb C}S$. As usual, we fix some positive $l$ and consider sections of class $W^l$ of the bundles.
\vspace{5pt}\\
Set
$$
\mathcal E=\{(E,\eta)\in\Gamma^l(\mathcal G)\times\Gamma^l(T^*S)\mid E+\bar E=T_{\mathbb C}S,\ \eta \text{ positive contact}\}
$$
Here by $\eta$ positive contact, we mean that $\eta$ is a contact form satisfying
\begin{equation}
\label{positivecontact}
d\eta(V,i\bar V)+d\eta(\bar V,-iV)>0
\end{equation}
for all non-zero vector of $E\cap D_{\mathbb C}$.
\vspace{5pt}\\
Observe that $\mathcal E$ is an open subset of the Hilbert space $\Gamma^l(\mathcal G)\times\Gamma^l(T^*S)$. 
Set now
$$
\mathcal I=\{(E,\eta)\in\mathcal E\mid [E,E]\subset E,\ \restriction{d\eta}{E}\equiv 0\}
$$
By corollary \ref{Echar}, the closed set $\mathcal I$ of $\mathcal E$ is exactly the set of sasakian structures of class $W^l$ on $S$.
\vspace{5pt}\\
Let $(E,\eta)\in\mathcal I$. A local chart for $\mathcal E$ at $(E,\eta)$ is given by
\begin{equation}
\label{Ecarte}
(\omega,\alpha)\in \Gamma^l((E^*\otimes D^{1,0})\oplus T^*S)
\mapsto ((Id-\omega)E,\eta+\alpha)\in\mathcal E
\end{equation}
Recall that
$$
T_{\mathbb C}S= D^{1,0}\oplus E=D^{1,0}\oplus D^{0,1}\oplus \mathbb C \xi.
$$
These three subbundles are involutive and correspond to foliated coordinates $(z,\bar z,t)$. In local foliated coordinates, we may thus decompose the de Rham differential as
$$
d=\partial+\bar\partial+\partial_t. 
$$
One can show that the operator $\bar\partial+\partial_t$ is indeed globally defined, whereas $\bar\partial$ and $\partial_t$ are not, cf. \cite{DK}.
\vspace{5pt}\\
For $\chi$ a smooth vector field, using the natural injection of $TS$ into $T_{\mathbb C}S$, we decompose it accordingly into
$$
\chi=\chi^{1,0}\oplus \chi^E=\chi^{1,0}\oplus\chi^{0,1}\oplus \chi^{\xi}
$$
where $\chi^{0,1}=\overline{\chi^{1,0}}$ and $\chi^\xi$ is real. 
\vspace{5pt}\\
The bundle $D^{1,0}$ is isomorphic to $N^{1,0}\F$ through the map \eqref{pi}, hence, in local foliated coordinates $(z,\bar z,t)$, is locally generated by the vector fields
$$
e_i=\dfrac{\partial}{\partial z_i}+a_i\xi
\leqno i=1,\hdots, n
$$
for some complex valued functions $a_i$. Such a field belongs to $D^{1,0}$ if it is in the kernel of $\eta$, hence we have
\begin{equation}
\label{ei}
e_i=\dfrac{\partial}{\partial z_i}-\eta\left (\dfrac{\partial}{\partial z_i}\right )\xi.
\end{equation}
Since $D^{1,0}$ is invariant under the flow of $\xi$, the transition functions of the bundle $D^{1,0}$ can be chosen as the transverse changes of charts of $\F$. Hence they are holomorphic and independent of $t$, and we may thus extend the operator $\bar\partial+\partial_t$ as a global operator acting on $(1,0)$-vector fields.
\vspace{5pt}\\
We are now in position to compute the integrability conditions and the differential of the action.
\begin{lemma}
\label{integrable}
The closed set $\mathcal I$ is locally isomorphic to the analytic set in $\Gamma^l((E^*\otimes D^{1,0})\oplus T^*S)$ given by the equations
\[\left\{\begin{array}{l}
(\bar\partial+\partial_t)\omega+\dfrac{1}{2}[\omega,\omega]=0\\
\\
Q(\omega,\alpha):=\restriction{\Big (d\alpha(Id-\omega,Id-\omega)-d\eta(\omega, Id)-d\eta(Id,\omega)\Big )}{E}\equiv 0.
\end{array}\right .\]
\end{lemma}

\begin{proof}
The first equation is the integrability condition of a transversely holomorphic foliation, see \cite{DK}. For the second one, by Corollary \ref{Echar}, it is given by
$$
\restriction{d(\eta+\alpha)}{(Id-\omega)E}\equiv 0.
$$
Using bilinearity, $\restriction{d\eta}{E}\equiv 0$ and $\restriction{d\eta}{D^{1,0}}\equiv 0$, we immediately obtain the result. 
\end{proof}

The group acting is just $\text{Diff}^{l+1}(S)$ with chart \eqref{Gcarte} given by the exponential associated to a fixed real analytic riemannian metric. Action \eqref{action} is
\begin{equation}
\label{saction}
(\omega,\alpha)\cdot \phi(v)=\left (\omega\cdot\phi(v) ,\phi(v)^*(\eta+\alpha)-\eta\right )
\end{equation}
where $\omega\cdot\phi(v)$ is characterized by 
\begin{equation}
\label{saction2}
\phi(v)_*\{w-(\omega\cdot \phi(v))(w)\mid w\in E\}=\{w-\omega(w)\mid w\in E\}
\end{equation} 
We have now
\begin{lemma}
\label{differentielle}
The differential $L$ of {\rm \eqref{saction}} at $(E,\eta)$ is
$$
L(\chi,\omega,\alpha)=(\omega+(\bar\partial+\partial_t)\chi^{1,0}, \alpha+\mathcal L_\chi \eta).
$$
\end{lemma}

\begin{proof}
The first component is computed in \cite{DK}, and the second one in Lemma \ref{Llemma}. 
\end{proof}

Following the notations of section \ref{local section}, we define the operator 
\begin{equation}
\label{D}
\chi\in\Gamma^l(TS)\longmapsto P(\chi)=((\bar\partial+\partial_t)\chi^{1,0},\mathcal L_\chi\eta)\in T_{(E,\eta)}\mathcal E.
\end{equation}

\begin{lemma}
\label{elliptic}
The operator $P$ is an elliptic differential operator of order $1$ from $TS$ into $(\Omega^1(E)\otimes D^{1,0})\oplus T^*S$.
\end{lemma}

Hence (H2diff) and (H3diff) are fulfilled.

\begin{proof}
From its definition \eqref{D}, $P$ is clearly a differential operator from $TS$ into $(\Omega^1(E)\otimes D^{1,0})\oplus T^*S$. Let us compute its symbol $\sigma$. Let $x\in S$ and $v\in T^*_xS\setminus\{0\}$. Choose local foliated coordinates $(z,\bar z,t)$, where we assume that $\xi=\partial/\partial t$. Then a direct computation shows that
\begin{equation}
\label{symbol}
\sigma_{(x,v)}(\chi)=(v^E\otimes\chi^{1,0},(i_\chi\eta) v).
\end{equation}
Assume now that $\sigma_{(x,v)}(\chi)$ is zero. Since $v$ is real and not zero, $v^E$ is not zero, so $\chi^{1,0}$ must be zero. This implies that $\chi^{0,1}$ is also zero, but it is not enough to conclude that $\chi$ is zero. 
\vspace{5pt}\\
However, looking at the second component of \eqref{symbol}, we have
$$
i_\chi\eta=\chi_\xi (i_\xi \eta)=\chi_\xi=0.
$$
This is exactly what was missing to conclude that $\chi$ is zero. Hence $\sigma_{(x,v)}(\chi)$ is injective and $P$ is elliptic.
\end{proof}

Finally, the automorphism group of a sasakian manifold is known to be a finite-dimensional Lie group and we can use the time $1$ flow as chart $\psi$ in (H4). We conclude from Proposition \ref{Pelliptic} and Lemma \ref{integrable} that, setting
\begin{equation}
\label{Kurgeneral}
K^l:=\{(\omega,\alpha)\mid P^*(\omega,\alpha)=(\bar\partial+\partial_t)\omega+\dfrac{1}{2}[\omega,\omega]=Q(\omega,\alpha)=0\}
\end{equation}
with $P^*$ the composition of the formal adjoint to $P$ and of chart \eqref{Ecarte}, we have

\begin{theorem}
\label{mainsasakian}
The infinite-dimensional analytic set {\rm \eqref{Kurgeneral}} is a Kuranishi type space for sasakian structures of class $W^l$ at $(E,\eta)$.
\end{theorem}

Moreover, each structure $(E,\eta)\in\mathcal E$ encodes a unique riemannian metric $g_{(E,\eta)}$ on $S$ as follows. Look at the second formula of \eqref{geta}. Starting from $(E,\eta)$ sasakian, it defines a riemannian metric. However, starting from $(E,\eta)$ only in $\mathcal E$, it does not give a symmetric expression. We claim that its symmetrization, that is
\begin{equation}
\label{gEeta}
g_{(E,\eta)}(V,W)=\dfrac{1}{4}(d\eta(V,\Phi(W))+d\eta(\Phi(V),W))+\eta(V)\eta(W)
\end{equation}
is a riemannian metric. Indeed it is definite positive on $E$ because of \eqref{positivecontact} and then on the whole $TS$ because of \eqref{Reebcontact}.
\vspace{5pt}\\
 This induces a unique riemannian 
metric
on the bundle $E^*\otimes D^{1,0}\oplus T^*S$, still denoted by $g_{(E,\eta)}$. By integrating over $S$, we obtain a scalar product $h_{(E,\eta)}$ on the space 
of global sections $\Gamma^l((E^*\otimes D^{1,0})\oplus T^*S)$. The collection $(h_\alpha)$ is a weak riemannian metric on $\mathcal E$, from which one deduces a strong riemannian metric. 
It is obviously invariant under the action of group of diffeomorphisms of $S$, and it is 
smooth by arguing as in the proof of Lemma \ref{H89}. Hypotheses (H6) and (H7) are thus satisfied. Hypothesis (H8) is only satisfied on the second component, but this is enough for the proof of Proposition \ref{Cinfini} to be applied (see \eqref{saction}). As an application of Theorem \ref{mainsmooth}, we thus have

\begin{corollary}
\label{mainsasakieninfini}
 Let $K^\infty$ be the subset of $C^\infty$ points of $K^l$. Then it is a Kuranishi type space for $C^\infty$ structures.
\end{corollary}

Observe that the equations in \eqref{Kurgeneral} are cubic, and not quadratic as in the classical case of complex structures. Let us compute more precisely
the adjoint $P^*$. This is similar to the computation of \eqref{Padjointeta}. Write
\begin{equation}
 \label{Pcompletbis}
 (h,\chi)\in\Gamma^{l+1}(\mathbb R\times D^{1,0})\longmapsto P(h,\chi)=((\bar\partial +\partial_t)\chi,i_{\chi+\bar\chi}d\eta+dh)\in\mathcal E
\end{equation}
and defining
$$
\Gamma_0:=\left (\left (\restriction{P}{\{0\}\times \Gamma^{l+1} (D^{1,0})}\right )^{-1}(P(\Gamma^{l+1} (\mathbb R)\times\{0\}))\right )^\perp
$$
then \eqref{Pcompletbis} becomes
\begin{equation}
 \label{Pcompletter}
 (h,\chi)\in\Gamma^{l+1}(\mathbb R)\times \Gamma_0\longmapsto P(h,\chi)=((\bar\partial +\partial_t)\chi,i_{\chi+\bar\chi}d\eta\oplus dh)\in\mathcal E.
\end{equation}
and we have
\begin{lemma}
 The adjoint of \eqref{Pcompletter} is given by the formula
 \begin{equation}
 \label{Padjoint}
  P^*(\omega,\alpha)=(d^*\alpha,(\bar\partial+\partial_t)^*\omega-(i_{\alpha^\sharp}d\eta)^\sharp) 
 \end{equation}
\end{lemma}

The proof is a direct computation and is completely similar to that of Lemma \ref{adjointeta}. Observe that $g$ being invariant by the flow of $\xi$, we may use the Hodge operator associated to $g$ to define both $d^*$ and $(\bar\partial+\partial_t)^*$.
\vspace{5pt}\\
As in Proposition \ref{sasakienH0=0}, we have

\begin{proposition}
Assume that the group of basic infinitesimal automorphisms $H^0(S,\Theta)$ is zero. Then $K^l$ and $K^\infty$ is a local moduli space.
\end{proposition}

\begin{proof}
If $H^0(S,\Theta)$ is zero, then by the semi-continuity theorems of \cite{DK}, it is zero also for $S'$ close to $S$. Hence, we may apply the proof of Proposition \ref{sasakienH0=0} and obtain that the automorphism group of any $S'$ in $K^l$ is equal to $\mathbb R \xi'$, and that (H4') and (H5') are satisfied. Theorem \ref{main} gives the result.
\end{proof}

\subsection{Comparison of the different deformation spaces} 

Let $S$ be a sasakian manifold. As Proposition \ref{charsasakien} suggests, it depends only on two structures: the transversely holomorphic foliation encoded in the subbundle $E$, and the contact form $\eta$.  We want to compare the associated three deformation spaces:

\begin{enumerate}
\item The Kuranishi type space $K^\infty_\eta$ of $\eta$-deformations defined in \eqref{Kspace}.

\item The Kuranishi type space $K^\infty$ of general deformations given in \eqref{Kurgeneral}.

\item The versal space $K_E$ of the transversely holomorphic foliation $(S,E)$.
\end{enumerate}

The space $K_E$ was obtained in \cite{GHS}. 
It is finite-dimensional and contains only smooth structures (so we drop the exponent since it is not relevant here).
\vspace{5pt}\\
Since $K_E$ is versal, there is a natural analytic map $\pi$ from $K^\infty$ to $K_E$ fixing $0$.
\vspace{5pt}\\
Observe that the set
$$
K_0^\infty:=\{(0,\alpha)\in K^\infty\}
$$
is exactly $K_\eta^\infty$ (up to shrinking). Hence we have a natural inclusion of $K_\eta^\infty$ into $K^\infty$. Of course, $K_0^\infty$ is included in $\pi^{-1}(\{0\})$, but there is no reason for this inclusion to be an equality in general. Nevertheless,
we have

\begin{proposition}
If $K^\infty$ is a local moduli space, then, up to shrinking, the central fiber $\pi^{-1}(\{0\})$ is equal to $K^\infty_\eta$.
\end{proposition}

\begin{proof}
Call $\mathcal I^\infty$ (respectively $\mathcal I_\eta^\infty$ and $\mathcal I_E^\infty$) the set of $C^\infty$ sasakian structures (respectively $C^\infty$ sasakian structures with fixed $E$ and $C^\infty$ sasakian structures with fixed $\eta$). Using the natural local encodings of these structures (see \eqref{Icontact} and \eqref{Ecarte}), we have a natural projection map
\begin{equation}
\label{PI}
(\omega,\alpha)\in \mathcal I^\infty\longmapsto \omega\in\mathcal I_E^\infty
\end{equation}
and isomorphisms onto their image
\begin{equation}
\label{iso1}
(\chi,\omega,\alpha)\in V\times K^\infty\longmapsto (\omega,\alpha)\cdot \phi(\chi)\in W^\infty\subset \mathcal I^\infty
\end{equation}
and 
\begin{equation}
\label{iso2}
(\chi,0,\alpha)\in V'\times K_0^\infty\longmapsto (0,\alpha)\cdot \phi'(\chi)\in 
(W')^\infty\subset \mathcal I_\eta^\infty .
\end{equation}
Be careful that $V$ is a neighborhood of $0$ in the Lie algebra of smooth vector fields of $S$, whereas $V'$ is a neighborhood of $0$ in the Lie algebra $\mathcal X$ defined in (\ref{vectTH}). Also the maps $\phi$ and $\phi'$ are not the same, cf. Lemma \ref{Banach}. Finally, we have an isomorphism
\begin{equation}
\label{iso3}
(\chi,\alpha)\in V''\times K_E\longmapsto \omega\cdot \phi(\chi)\in (W'')^\infty\subset \mathcal I_E^\infty 
\end{equation}
where $V''$ is a neighborhood of $0$ in the Lie algebra of smooth vector fields of $S$. We assume, restricting the Kuranishi spaces if necessary, that the image
\begin{equation}
\label{assum}
\{\phi'(\chi')\circ\phi(\chi)\mid \chi\in V'',\ \chi'\in V'\}
\end{equation}
is included in $\phi(V)$.
\vspace{5pt}\\
Now, let $(\omega,\alpha)\in\pi^{-1}(0)$. By \eqref{iso3}, that means that there exists some $\chi\in V''$ such that 
$$
\omega\cdot \phi(\chi)=0.
$$
Set
$$
(0,\alpha_0):=(\omega,\alpha)\cdot \phi(\chi).
$$
Using \eqref{iso2}, we know that there exists $\chi'\in V'$ such that $(0,\alpha_0)\cdot \phi'(\chi')$ belongs to $K_0^\infty$. Hence,
$$
(0,\alpha_1):=(0,\alpha_0)\cdot \phi'(\chi')=(\omega,\alpha)\cdot (\phi'(\chi')\circ\phi(\chi))
$$
so $(0,\alpha_1)$ and $(\omega, \alpha)$ both belong to $K^\infty$ and represent the same sasakian structure. Moreover, they belong to the same local orbit of $\text{Diff}(S)$ in $W^\infty$. Now condition \eqref{assum} associated to \eqref{iso1} and the local moduli space assumption shows that
$$
(0,\alpha_1)=(\omega,\alpha).
$$
As $(0,\alpha_1)$ belongs to $K_0^\infty$, this implies 
$$
\pi^{-1}(0)\subset K_0^\infty.
$$
Since we already noticed that the other inclusion is clear, we are done.
\end{proof}

\begin{remark}
The map $\pi$ is not surjective. Indeed, consider the vector field
$$
\chi_\lambda =z\dfrac{\partial}{\partial z}+\lambda w\dfrac{\partial}{\partial w}
$$
in $\mathbb C^2$. For $\lambda\in \mathbb C\setminus (-\infty,0]$, the flow of $\chi_\lambda$ is transverse to the unit sphere $\mathbb S^3$ and induces a transversely holomorphic flow on it. It is known however that there exist a sasakian metric associated to $\chi_\lambda$ if and only if the flow is riemannian, and this happens exactly 
when $\lambda$ is real 
(see \cite{BoyerBook}).
\end{remark}

\subsection{Deformations of Sasaki-Einstein manifolds}
\label{Sasaki-Einstein}

One of the advantages of our setting is that, given a Kuranishi type space $K$ for a certain class of structures, we can easily deduce a Kuranishi type space for more specific structures. It is only a matter of adding integrability conditions both in the definition of $\mathcal I$ and of $K$.\vspace{5pt}\\
In this subsection, we play this game with Sasaki-Einstein manifolds viewed as special sasakian manifolds.\vspace{5pt}\\
Recall that a sasakian manifold $(S, E,\eta)$ is {\it Sasaki-Einstein} if its sasakian metric $g$ satisfies \cite[\S 1.4]{Sparks}
$$
\text{Ric}_g=(\dim S -1)g.
$$
Starting with $(S,E,\eta)$ Sasaki-Einstein, we immediately obtain from \ref{KuranishiEfixed}, \ref{KuranishiEbisfixed} and \ref{mainsasakian} the following statement.\vspace{5pt}\\
Let 
\begin{equation}
\label{KSE}
K^{SE}=\{(\omega,\alpha)\in K^\infty\mid \text{Ric}_g=(\dim S -1)g\}
\end{equation}

\noindent and define similarly $K^{SE}_\eta$ and $(K')^{SE}_\eta$ from $K^\infty_\eta$ and $(K')^{\infty}$. With these notations,

\begin{corollary}
\label{mainSE}
We have
\begin{enumerate}
\item The space $K^{SE}$ is a Kuranishi type space for smooth Sasaki-Einstein manifolds close to $(S,E,\eta)$. 
\item The space $K^{SE}_\eta$ is a Kuranishi type space for smooth Sasaki-Einstein manifolds close to $(S,E,\eta)$ with the transversely holomorphic structure $E$ fixed. 
\item The space $(K')^{SE}_\eta$ is a Kuranishi type space for smooth Sasaki-Eins\-tein manifolds close to $(S,E,\eta)$ with the transversely holomorphic 
structure $E$  and the Reeb vector field $\xi$ fixed. 
\end{enumerate}
\end{corollary}

We observe that, following \cite{NS}, the space $(K')^{SE}_\eta$ can be identified with a neighborhood of the identity in the automorphism group of the transverse holomorphic structure of $S$.

\end{document}